\documentclass[11pt]{article}

\usepackage{amssymb,epsf}
\usepackage{amssymb,amsmath}
\usepackage{comment}

\newtheorem{theorem}{Theorem}
\newtheorem{proposition}{Proposition}
\newtheorem{lemma}{Lemma}
\newtheorem{definition}{Definition}
\numberwithin{equation}{section}

\def\argmin{\mathop{\rm arg\,min}\limits}

\def\Ps{\Psi}
\def\R{\mathcal{R}}
\def\free{z}

\def\MR{MR}

\begin{document}

\title{Adaptive spectral regularizations of high dimensional linear models
\thanks{The author  was partially  supported  by the ANR grant no. ANR-BLAN-0234-01 and  by
Laboratory of Structural Methods of Predictive Modeling and Optimization, MIPT, RF government grant,
ag. 11.G34.31.0073.}}

\author{ Yuri Golubev  \thanks{
Universit\'e de Provence, Marseille, France and
Moscow Institute of Phisics and Technology, Russia.} \\ \texttt{golubev.yuri@gmail.com}}
\date{}

\maketitle

\begin{abstract}
This paper  focuses on recovering  an unknown vector $\beta$
from the noisy data $Y=X\beta +\sigma\xi$, where $X$ is a known
$n\times p$-matrix, $\xi $ is a standard white Gaussian noise,
and $\sigma$ is an unknown noise level. In
order to estimate $\beta$,  a spectral regularization method is used,
and our goal is to choose its regularization parameter  with
the help of the data $Y$. In this paper, we deal  solely with
regularization methods based on  the so-called ordered smoothers
(see  \cite{K})  and  extend  the oracle inequality from  \cite{G1} to the case,
where the noise level is unknown.
\end{abstract}

\section{Introduction and main results}\label{s1}
This paper deals with
  recovering an unknown vector $\beta\in \mathbb{R}^n$ from the noisy data
\begin{equation*}
Y=X\beta +\sigma\xi,
\end{equation*}
where $X$ is a known $n\times p$-matrix with $n\ge p$, $\xi=\bigl(\xi(1),\ldots,\xi(n)\bigr)^\top
$ is a standard white Gaussian noise
($\mathbf{E}\xi(k)=0,\, \mathbf{E}\xi^2(k)=1, \, k=1,\ldots,n$
), and  $\sigma$ is an unknown   noise level.

 In spite of its simplicity, this mathematical model plays an important role  in solving
     practical  inverse problems  like  gravity problems (see, e.g.
 \cite{BMP}), tomography inverse problems \cite{H}, and many others. As a rule,  in  inverse problems
 $n$ and $p$ are very large and therefore the main goal in this paper is to propose an approach suitable for
$n=\infty,\, p=\infty$, severely ill-posed matrices $X^\top X$, and the unknown noise level.

  We begin with the standard maximum likelihood estimate
 \[
\hat{\beta}_0=\argmin_{\beta\in \mathbb{R}^p}\|Y-X\beta\|^2= (X^{\top}X)^{-1}X^{\top}Y ,
 \]
where $ \|z\|^2=\sum_{k=1}^nz^2(k).
 $ It is well known and easy to check  that \[\mathbf{E}(\hat\beta_0-\beta) (\hat\beta_0-\beta)^\top=\sigma^2 (X^\top X)^{-1}\] and thus,  the mean square risk
of $\hat\beta_0$ is computed as follows:
\begin{equation}\label{eqn.1}
\mathbf{E} \|\hat\beta_0-\beta\|^2 =\sigma^2{\rm trace }\Bigl[(X^\top X)^{-1}\Bigr]  =\sigma^2\sum_{k=1}^p
\lambda^{-1}(k),
\end{equation}
where $\lambda(k)$ and $\psi_k\in \mathbb{R}^n$ here and below are  the eigenvalues
and the eigenvectors of $X^{\top}X$
\[
 X^{\top}X\psi_k =\lambda(k) \psi_k.
\]

In this paper, it is assumed  solely that
$\lambda(1) \ge \lambda(2) \ge \cdots\ge \lambda(p)$. This assumption together with
 (\ref{eqn.1}) reveals the main difficulty in $\hat{\beta}_0$:
\textsl{its risk may be very large when $p$ is large or when $X^\top X$ has
a large condition number.}

The natural idea  to improve  $\hat\beta_0$    is to suppress
large $\lambda^{-1}(k)$ in (\ref{eqn.1}) with the help of a linear smoother. Therefore  we make  use of  the following family of linear  estimates
\begin{equation}\label{eqn.2}
\hat{\beta}_\alpha=H_\alpha\hat\beta_0,\,  \alpha\in(0,\alpha^\circ] ,\end{equation}
 where $H_\alpha,\,  \alpha\in(0,\alpha^\circ]$ is a family of $p\times p$-smoothing matrices.

In what follows,
  we deal with the  smoothing matrices admitting the following representation \[
H_\alpha=\sum_{k=1}^p \mathcal{H}_\alpha[\lambda(k)]\psi_k\psi_k^\top,
\]
where  $\mathcal{H}_\alpha(\lambda): \ \mathbb{R}^+\rightarrow [0,1]$ is  such that $$
\lim_{\alpha\rightarrow 0}\mathcal{H}_\alpha(\lambda)=1,\quad \lim_{\lambda\rightarrow
0}\mathcal{H}_\alpha(\lambda)=0. $$

In the literature (see, e.g., \cite{EHN}), this method is called {\it spectral
regularization}. It
covers widely used  regularizations methods such as the Tikhonov-Phillips regularization \cite{TX} known in the statistical literature as \textit{ridge regression}, Landweber's iterations \cite{L},
the $\mu$-method (see e.g. \cite{EHN}), and many others.

 Summarizing,  $\beta$ is estimated with the help of  the  family of linear estimates $\hat{\beta}_\alpha,  \ \alpha\in(0,\alpha^\circ]$ defined by (\ref{eqn.2})
and our goal is to find based on the data at hand the best estimator within this family. Notice that for given $\alpha$,   the mean square risk of $\hat \beta_\alpha$  is computed as follows:
\begin{equation} \label{eqn.3}
\begin{split}
L_\alpha(\beta)\,\stackrel{\rm def}{=}\,&\mathbf{E}\|\hat
\beta_\alpha-\beta\|^2 =\sum_{k=1}^p\bigl[1-
h_\alpha(k)\bigr]^2\langle\beta,\psi_k\rangle^2+\sigma^2
\sum_{k=1}^p \lambda^{-1}(k) h_\alpha^2(k),
\end{split}
\end{equation}
where   $$ h_\alpha(k)\stackrel{\rm def}{=}\mathcal{H}_\alpha[\lambda(k)]\quad  \text{and} \quad
\langle\beta,\psi_k\rangle \stackrel{\rm def}{=}\sum_{l=1}^p \beta(l)\psi_k(l). $$
It is easily seen from (\ref{eqn.3}) that the variance of $\hat \beta_\alpha$ is always smaller than that one of $\hat \beta_0$, but  $\hat \beta_\alpha$ has a non-zero bias and therefore adjusting   $\alpha$ we may  improve the risk of $\hat\beta_0$. This improvement may be very significant
 when $\langle\beta,\psi_k\rangle^2$ are small for
large $k$.

In practice, a good choice of the regularizing matrix family $H_\alpha$ is a delicate
problem related to the computational complexity of $\hat\beta_\alpha$. For details, we refer interested readers to \cite{EHN}.

As a rule,  practical spectral   regularization methods (the spectral cut--off, the Tikhonov-Phillips
regularization,   Landweber's
iterations)  represent    the so-called  \textsl{ordered smoothers}  \cite{K}. This means that the family of functions $\{\mathcal{H}_\alpha(\lambda),\,  \alpha\in (0,\alpha^{\circ}]\}$ is ordered in the following
sense:
\begin{definition}\label{d1}
The family of functions $\{{F}_\alpha(\lambda),\,  \alpha\in A, \, \lambda \in\Lambda\subseteq\mathbb{R}^+\}$ is
 ordered if:
\begin{enumerate}
\item  For any given $\alpha\in A$, $ F_\alpha(\lambda): \ \Lambda\rightarrow [0,1]$ is  a monotone function of $\lambda$.
\item  If for some $ \alpha_1,\alpha_2\in A$ and some $ \lambda'\in \Lambda$,
 $F_{\alpha_1}(\lambda') < F_{\alpha_2}(\lambda')$,
 then for all
$  \lambda\in \Lambda$,  $
 F_{\alpha_1}(\lambda) \le  F_{\alpha_2}(\lambda)$.
\end{enumerate}
\end{definition}

 The next important  question usually arising in practice is related to the data-driven choice of
 the regularization parameter $\alpha$.    In statistical literature, one can  find several general approaches to this problem. We cite here, for instance,
the Lepski method which has been  adopted  to inverse problems in \cite{Ma}, \cite{BH}, \cite{BHMR}, and the model selection technique which was used in \cite{LL}.

The approach proposed in this paper is  a slight modification of the unbiased risk estimation. To make the presentation simpler, we begin with  the case, where the noise level $\sigma^2$ is known.
Intuitively, a good data-driven regularization  parameter should minimize in some sense the risk
$L_\alpha(\beta)$ (see
(\ref{eqn.3})). Obviously, the best regularization parameter  minimizing   $L_\alpha(\beta)$
cannot be used since it depends on the unknown parameter of interest $\beta$.
However, the idea of minimization of $L_\alpha(\beta)$  may be  put into practice with
 the help of  the
empirical risk minimization principle  defining  the
regularization parameter  as follows:
\begin{equation}\label{eqn.4}
\hat\alpha=\argmin_\alpha {R}_{\alpha}^\sigma[Y,Pen],
\end{equation}
where $$ {R}_\alpha^\sigma[Y,Pen]\stackrel{\rm def}{=}
\|\hat\beta_0-\hat\beta_\alpha\|^2+\sigma^2Pen(\alpha),$$
 and $Pen(\alpha): (0,\alpha^{\circ}]\rightarrow\mathbb{R}^+$ is a given function  called \textsl{penalty}.
 The main idea in this approach is to link  $L_\alpha(\beta)$ and  ${R}_\alpha^\sigma[Y,Pen]$. Heuristically, we want  to find
    a minimal penalty $Pen(\alpha)$ that ensures the following inequality
 \begin{equation}\label{eqn.5}
 L_\alpha(\beta)\lesssim
{R}_{\alpha}^\sigma[Y,Pen]+\mathcal{C},
\end{equation}
where $\mathcal{C}$ is a random variable that doesn't depend on $\alpha$.
It is convenient to define this constant as follows:
$$
\mathcal{C}=-\|\beta-\hat\beta_0\|^2=-\sigma^2\sum_{k=1}^p \lambda^{-1}(k)\xi^2(k).
$$

The traditional approach to solving (\ref{eqn.5}) is based on the minimization of the unbiased risk estimate. In this method,   the penalty is computed  as a root of the equation
\begin{equation} \label{eqn.6}
L_\alpha(\beta)=
\mathbf{E}\Bigl\{{R}_{\alpha}^\sigma[Y,Pen_u]+\mathcal{C}\Bigr\}.
\end{equation}
One can check with a simple algebra that
\[
Pen_{u}(\alpha)=2 \sum_{k=1}^p
\lambda^{-1}(k) h_\alpha(k).
\]
The idea of this penalty goes back to \cite{A} and  \cite{CGPT} provides some oracle inequalities related to  this approach.

Another well-known and widely used approach to the data-driven choice of $\alpha$  is  related to the cross validation technique
\cite{DMR}. In the framework of our statistical model,
this method prompts a data-driven regularization parameter which is close to
\[
\hat{\alpha}_{CV}=\arg\min_{\alpha}\Bigl\{\|Y-X\hat{\beta}_\alpha\|^2+
\sigma^2Pen_{CV}(\alpha)\Bigr\},
\]
with
\[
Pen_{CV}(\alpha)=2 \sum_{k=1}^p
 h_\alpha(k).
\]
It is well-known (see e.g. \cite{K}) that if the risk of $\hat{\beta}$
is measured by $\mathbf{E}\|X\hat{\beta}-X\beta\|^2$,  then this penalty is nearly optimal and it works always well.

However, the question \textit{Does $\hat{\alpha}_{CV}$ works well when  the risk is measured
by $\mathbf{E}\|\hat{\beta}-\beta\|^2$  ?} has a
delicate answer depending  on the spectrum of $X^\top X$. To the best of our knowledge there are no  oracle inequalities controlling  the risk of $\hat{\beta}{\hat{\alpha}_{CV}}$ uniformly in $\beta$.
Notice, however, that one can show with the help of the method for computing minimal penalties    in \cite{BM}, that if $\lambda(k)\le \exp(-\kappa k)$, then
  the risk of this method  blows up starting from  some $\kappa>0$.

The similar effect takes place in the unbiased risk estimation.
This happens because the standard deviation of ${R}_{\alpha}^\sigma[Y,Pen_{u}]+\mathcal{C}$
may be very large with respect to the mean  $\mathbf{E}\bigl\{{R}_{\alpha}^\sigma[Y,Pen_{u}]+\mathcal{C}\big\}$
and therefore (\ref{eqn.5}) may fail with a high probability.

To improve the above mentioned  drawbacks of the unbiased risk estimation,   we define, following \cite{G1}, the penalty as a minimal root of the equation
\begin{equation}\label{eqn.7}
\begin{split}
&\mathbf{E}\sup_{ \alpha\le \alpha^{\circ}} \Bigl[L_\alpha(\beta)-
{R}_{\alpha}^\sigma[Y,Pen]-\mathcal{C}\Bigr]_+  \le C_1\mathbf{E} \Bigl[ L_{\alpha^{\circ}}(\beta)-
{R}_{\alpha^{\circ}}^\sigma[Y,Pen]-\mathcal{C}\Bigr]_+,
\end{split}
\end{equation}
 where $[x]_+=\max\{0,x\}$ and
$C_1>1$ is a constant. Heuristic motivation behind this approach is rather transparent. We are looking for the minimal penalty that balances the excess risks corresponding to all possible  $\alpha\in (0,\alpha^{\circ}]$.  Recall that the excess risk is defined by the  difference between the  risk of the estimate and its penalized  empirical risk. Note that in view of (\ref{eqn.5}), we can deal  solely with    the positive part of the excess  risk.

   In order to explain heuristically  how Equation  (\ref{eqn.7}) may be solved,  we begin with the spectral representation of the underlying statistical problem.
 One can check easily that
\begin{equation*}
y(k)\stackrel{\rm def }{=} \langle X^\top Y,\psi_k\rangle/ \lambda(k) =
\langle \beta,\psi_k\rangle +\sigma \xi'(k)/\sqrt{\lambda(k)},
\end{equation*}
where  $\xi'(k)$ are i.i.d. $\mathcal{N}(0,1)$. With these notations,
$\hat\beta_\alpha$ admits the following representation
\[
\langle \hat\beta_\alpha,\psi_k\rangle =h_\alpha(k)y(k)=h_\alpha(k)\beta(k)+\sigma h_\alpha(k)\xi'(k)/\sqrt{\lambda(k)},
\]
where $\beta(k)=\langle \beta,\psi_k\rangle$, and therefore
\begin{equation} \label{eqn.8}
\begin{split}
\|\hat\beta_0-\hat\beta_\alpha\|^2&=\sum_{k=1}^p\bigl[1-h_\alpha(k)\bigr]^2
y^2(k),\\
\|\beta-\hat\beta_\alpha\|^2&=\sum_{k=1}^p\bigl[\beta(k)-h_\alpha(k)y(k)\bigr]^2.
\end{split}
\end{equation}

In what follows, it is assumed that the penalty  has the following structure
   \begin{equation*}
   Pen(\alpha) =Pen_{u}(\alpha)+(1+\gamma)Q(\alpha),
   \end{equation*}
   where $\gamma$ is a small positive number and $Q(\alpha), \, \alpha>0$  is a positive  function of $\alpha$ to be defined later on. Recall that the first term at the right-hand side
   is  obtained from the unbiased risk estimation (see Equation (\ref{eqn.6})).
 With $Pen(\alpha)$ we can rewrite  the excess risk  as follows:
 \begin{align}\label{eqn.9}
 \begin{split}
 & L_\alpha(\beta)-{R}_{\alpha}^\sigma[Y,Pen]-\mathcal{C}\\
&\quad =\sigma^2 \sum_{k=1}^p \lambda^{-1}(k)\bigl[2h_\alpha(k)-h_\alpha^2(k)\bigr](\xi'^2(k)-1)-(1+\gamma)
\sigma^2Q(\alpha)\\
&\qquad +2\sigma  \sum_{k=1}^p \lambda^{-1/2}(k)\bigl[2h_\alpha(k)-h_\alpha^2(k)\bigr]\xi'(k)\beta(k).
\end{split}
 \end{align}

The first idea in solving (\ref{eqn.7}) is based on the  the fact  that the cross term
 \[
 2\sigma  \sum_{k=1}^p \lambda^{-1/2}(k)\bigl[2h_\alpha(k)-h_\alpha^2(k)\bigr]\xi'(k)\beta(k)
 \]
 is typically  small with respect to $\mathbf{E}\bigl\{R_\alpha^\sigma[Y,Pen]+\mathcal{C}\bigr\}$ (see for more details  Lemma 9 in \cite{G1}). With this in mind, omitting the cross term,  Equation  (\ref{eqn.7}) can be rewritten in
  the following nearly equivalent  form
 \begin{equation}\label{eqn.10}
 \mathbf{E}\sup_{\alpha\le \alpha^{\circ}}[\eta_\alpha-(1+\gamma)Q(\alpha)]_+ \lesssim C_1\mathbf{E}[\eta_{\alpha^{\circ}}-(1+\gamma)Q(\alpha^{\circ})]_+ \asymp D(\alpha^{\circ}),
 \end{equation}
 where
 \[
 \eta_\alpha\stackrel{\rm def }{=}\sum_{k=1}^p \lambda^{-1}(k)\bigl[2h_\alpha(k)-h_\alpha^2(k)\bigr][\xi'^2(k)-1]
 \]
 and
 \[
 D(\alpha)\stackrel{\rm def }{=} \sqrt{\mathbf{E}\eta^2_\alpha}=\biggl\{2\sum_{k=1}^p \lambda^{-2}(k)\bigl[2h_\alpha(k)-h_\alpha^2(k)\bigr]^2\biggr\}^{1/2}.
 \]

 Now we are in a position to compute an approximation of the minimal root for (\ref{eqn.10}). It is clear that $Q(\alpha)\ge Q^+(\alpha) $, where $ Q^+(\alpha) $ is a root of
\begin{equation}\label{eqn.11}
\mathbf{E}[\eta_\alpha-Q^+(\alpha)]_+=D(\alpha^{\circ}).
\end{equation}

To find a feasible solution to (\ref{eqn.11}), we make use of  the exponential Chebychev inequality resulting in
\begin{equation}\label{eqn.12}
\mathbf{E}[\eta-x]_+^p\le \Gamma(p+1)\lambda^{-p}\exp(-\lambda x)
\mathbf{E}\exp(\lambda\eta),
\end{equation}
where $\eta$ is a random variable, $\Gamma(\cdot)$ is the gamma function, and
  $\lambda>0$.

Therefore we define  $Q^+(\alpha)$ as a root of equation
\begin{equation*}
\inf_{\lambda}\exp[-\lambda Q^+(\alpha)]\mathbf{E}\exp(\lambda \eta_\alpha)=D(\alpha^{\circ}).
\end{equation*}

It is easy to check with a simple algebra that
\begin{equation}\label{eqn.13}
Q^+(\alpha)=2D(\alpha)\mu_\alpha\sum_{k=1}^p\frac{\rho^2_\alpha(k)}{1-2\mu_\alpha \rho_\alpha(k)},
\end{equation}
where  $\mu_\alpha$ is a root of the equation
 \begin{equation}\label{eqn.14}
 \sum_{k=1}^p F[\mu_\alpha \rho_\alpha(k)]=\log \frac{D(\alpha)}{D(\alpha^{\circ})},
 \end{equation}
 and
 \begin{equation}\label{eqn.15}
 \begin{split}
 F(x)&=\frac{1}{2}\log(1-2x)+x +\frac{2x^2}{1-2x},\\
 \rho_\alpha(k)&=\sqrt{2}D^{-1}(\alpha)\lambda^{-1}(k)\bigl[2h_\alpha(k)
 -h_\alpha^2(k)\bigr].
 \end{split}
 \end{equation}

 The next result  (see  also Theorem 1 in \cite{G1}) shows that  $Q^+(\alpha)$ is a nearly optimal solution to (\ref{eqn.10}).
\begin{proposition}  \label{proposition.1}  For any $\gamma>q\ge 0$
\begin{equation*}
\mathbf{E}\sup_{\alpha\le \alpha^{\circ}}\Bigl[\eta_\alpha-(1+\gamma)Q^+(\alpha)\Bigr]_+^{1+q}\le \frac{CD^{1+q}(\alpha^{\circ})}{(\gamma-q)^3},
\end{equation*}
where here and throughout the paper $C$ denotes a generic constant.
\end{proposition}

Let us now  turn  to the case, where $\sigma$ is unknown. To compute the data-driven regularization parameter  in this situation,
we replace $\sigma^2$ in  $R^\sigma_\alpha[Y,Pen]$ by  the standard variance  estimator
\[
 \hat{\sigma}^2_\alpha=\frac{\|Y-X\hat\beta_{\alpha}\|^2}{\|1-H_{\alpha}\|^2}.
 \]
 Thus we arrive at the following approximation of the empirical risk
\begin{equation*}
R_\alpha[Y,Pen]\stackrel{\rm def}{=}
\|\hat\beta_0-\hat\beta_\alpha\|^2+\frac{\|Y-X\hat\beta_\alpha\|^2}{\|1-H_\alpha\|^2}Pen(\alpha)
\end{equation*}
and the data-driven regularization parameter is computed now as follows:
\begin{equation*}
\hat\alpha =\argmin_{\alpha_\circ\le\alpha\le\alpha^\circ}R_\alpha[Y,Pen].
\end{equation*}

Notice that in contrast to the case of known $\sigma$, it assumed that $\alpha$ is bounded from below by $ \alpha_\circ$. This constraint ensures  that with a hight probability $[\sigma^2-\hat\sigma_{\hat{\alpha}}^2]_+< \sigma^2/2$ uniformly in $\beta\in \mathbb{R}^p$. Unfortunately, when this inequality fails we cannot control
correctly the risk of $\hat\beta_{\hat\alpha}$ since it may blow up (see \cite{BM} for similar phenomenon in the model selection).
  So, to  avoid the blowup, we need a relatively good estimate
of $\sigma$, or equivalently, large $\|1-H_\alpha\|^2$.

Stress also that since $\alpha_\circ$ cannot  depend on $\sigma$, we would like
to have $\alpha_\circ$ as small as possible  to be sure that the methods works for  small noise levels. From a mathematical viewpoint, this means that we need a relatively
good upper bound for $\mathbf{E}| \sigma-\hat\sigma_{\hat{\alpha}}^2|Pen(\hat{\alpha})$. Roughly speaking, we have to check that with a hight probability
\[
| \sigma^2-\hat\sigma_{\hat{\alpha}}^2|Pen(\hat{\alpha})\ll\sigma^2 Pen(\hat{\alpha}).
\]
The main difficulty in proving this equation is related to the fact that the random variables $\sigma^2-\hat\sigma_{\hat{\alpha}}^2$ and $ Pen(\hat{\alpha})$ are dependent. To overcome this difficulty we make use of
the law of the iterated  logarithm for $\sigma^2-\hat\sigma_{\hat{\alpha}}^2$  combined with a generalization of the
H\"{o}lder inequality (see Lemmas \ref{lemma.4} and \ref{lemma.3} below).
To carry out this approach, we need the following additional condition: \textit{there exists a
positive constant $C_2$ such that for all $\alpha\in (0,\alpha^\circ]$
\begin{align}\label{eqn.16}
&\|h_\alpha\|^2_\lambda \ge C_2 \sum_{k=1}^p \lambda^{-1}(k)h_\alpha(k), \\[5pt]
\label{eqn.17}
&\frac{\|h_\alpha\|^2_\lambda}{\log[D(\alpha)/D(\alpha^\circ)]}+\max_k \frac{h_\alpha(k)}{\lambda(k)}\ge C_2 D(\alpha),
\end{align}
where
\[
\|h_\alpha\|^2_\lambda =\sum_{k=1}^p \lambda^{-1}(k)h^2_\alpha(k).
\]
}

Denote for brevity
\begin{equation*}
\begin{split}
 \Ps(\alpha_\circ,\alpha^\circ) \stackrel{\rm def}{=}
\frac{1}{\|1-h_{\alpha_\circ}\|}\biggl\{\biggl[\log\log\biggl(1+\frac{\|1-h_{\alpha_\circ}\|^2}{\|1-h_{\alpha^\circ}\|^2}\biggr)\biggr]^{1/2}
\\ {}+\log \biggl(1+\frac{Pen(\alpha_\circ)}{Pen(\alpha^\circ)}\biggr)\biggr\}
.
\end{split}
\end{equation*}

The following theorem   controls the risk of  $\hat\beta_{\hat\alpha}$ via the  penalized oracle risk defined by
 \[
 r(\beta)\stackrel{\rm def}{=}\inf_{\alpha_\circ\le \alpha\le \alpha^{\circ}} \bar{R}_{\alpha}(\beta),
 \]
 where
\begin{equation*}
\begin{split}
 \bar{R}_{\alpha}(\beta)\,\stackrel{\rm def}{=}\,\mathbf{E}_\beta\bigl\{
R_\alpha[Y,Pen]+\mathcal{C}\bigr\} = L_\alpha(\beta) +(1+\gamma)\sigma^2Q^+(\alpha)\\ {}+ \frac{Pen(\alpha)}{\|1-h_\alpha\|^2}\sum_{k=1}^p \bigl[1-h_\alpha(k)\bigr]^2 \lambda(k) \beta^2(k).
\end{split}
\end{equation*}

 \begin{theorem}\label{theorem.1} Let
 $ Pen(\alpha) =2 \sum_{k=1}^p
\lambda^{-1}(k) h_\alpha(k)+(1+\gamma)Q^+(\alpha)
$ with $Q^+(\alpha)$ defined by (\ref{eqn.13}--\ref{eqn.15}) and suppose (\ref{eqn.16}--\ref{eqn.17}) hold.
Then, uniformly in $\beta\in \mathbb{R}^p$,
 \begin{align}\label{eqn.18}
 \begin{split}
\mathbf{E}_\beta\|\beta -\hat \beta_{\hat\alpha}\|^2\le
\biggl[1+C\Ps(\alpha_\circ,\alpha^\circ)+ C\log^{-1/2}\frac{ r(\beta)}{\sigma^2 D(\alpha^{\circ})}\biggr]
r(\beta)\\
{}+\frac{C\sigma^2D(\alpha^\circ)}{[1-C\Ps(\alpha_\circ,\alpha^\circ)/\gamma]_+\sqrt{\gamma} }\R\biggl[\frac{r(\beta)}{\sigma^2\gamma D(\alpha^{\circ})}+ \frac{1}{\gamma^4}\biggr]
,
 \end{split}
 \end{align}
where
$
\R(x)={x}/{\log(x)}.
$
\end{theorem}

Notice that  Equation \ref{eqn.18} can be rewritten  in the following concise form
\begin{equation}\label{eqn.19}
\mathbf{E}_\beta\|\beta -\hat \beta_{\hat\alpha}\|^2\le  \biggl[1+C\Ps(\alpha_\circ,\alpha^\circ)+ \Psi_{\alpha_\circ,\gamma}\biggl(\frac{ r(\beta)}{\sigma^2 D(\alpha^{\circ})}\biggr)\biggr]r(\beta),
\end{equation}
where $\Psi_{\alpha_\circ,\gamma}(\cdot)$ is a bounded function such that
\begin{equation}\label{eqn.20}
\lim_{x\rightarrow \infty}\Psi_{\alpha_\circ,\gamma}(x)=0.
\end{equation}

The statistical sense of (\ref{eqn.19}) is rather transparent: this equation shows that in typical nonparametric  situations the method  works   like the ideal penalized oracle with the risk $r(\beta)$. \textit{The typical nonparametric situation} means that
\begin{itemize}
\item
 $p$ is large, so,  for properly chosen  $\alpha_\circ$, $ \Ps(\alpha_\circ,\alpha^\circ)$ is small,
 \item the vector $(\langle\beta,\psi_1\rangle,\ldots \langle\beta,\psi_p\rangle)^\top$ contains many significant  components, and thus $r(\beta)\gg \sigma^2D(\alpha^\circ)$.
\end{itemize}

These assumptions are typical in the minimax estimation, where it is assumed that $\beta$ belongs to an
ellipsoid.  Notice that with the help of  (\ref{eqn.19}--\ref{eqn.20})
one can check relatively easily   that for a proper chosen  spectral regularization, $\hat\beta_{\hat\alpha}$ is the asymptotically minimax estimate up to a constant (see for details \cite{G1} and \cite{P}).

  We finish this section with a short discussion of  Conditions (\ref{eqn.16}--\ref{eqn.17}).
  Equation (\ref{eqn.16})  means that $h_\alpha(k)$ vanishes rather rapidly for large $k$. This is always true for the spectral cut-off method
($h_\alpha(k)=\mathbf{1}\{\alpha \lambda(k)\ge 1\}$).
Indeed,  if
 $\lambda(k)\asymp k^{-p}$ with some $p\ge 0$, then
 \[
 \|h_\alpha\|^2_\lambda \asymp {\alpha}^{-p-1}, \quad
 D(\alpha) \asymp {\alpha}^{-p-1/2}
 \]
 and it is seen easily that (\ref{eqn.17}) is fulfilled. Assume now that $X^\top X$ is severely ill-posed, i.e.,  $\lambda(k)\asymp
 \exp(-\kappa k)$ with $\kappa>0$. Then
 \[
 \max_k \lambda(k)h_\alpha(k)\asymp \exp(\kappa/\alpha) \quad \text{and}\quad
 D(\alpha) \asymp \kappa^{-1/2}\exp(\kappa/\alpha).
 \]
 Therefore (\ref{eqn.17}) holds with $C_2 = \kappa^{-1/2}$.

\section{Proofs}

\subsection{Ordered processes and their basic properties}
  The main results in this paper are based on a  general fact which is similar to Dudley's entropy bound (see, e.g., \cite{VW}).
Let $\zeta_t$ be a separable zero mean random  process on $ \mathbb{R}^+$.
Denote for brevity
\[
\Delta^\zeta(t_1,t_2)=\zeta_{t_1}-\zeta_{t_2}.
\]
The following fact (see Lemma 1 in \cite{G1}) plays a cornerstone role in the proof of Proposition \ref{proposition.1} and Theorem \ref{theorem.1}.
\begin{proposition}\label{proposition.2} Let  $v^2_u,\ u\in \mathbb{R}^+$,  be a continuous strictly increasing function with $v^2_0=0$. Then for any $\lambda>0$,
\begin{equation*}
\begin{split}
\log \mathbf{E}\exp\biggl\{\lambda \max_{0\le s\le
t}\frac{\Delta^\zeta(s,t)}{\sigma_t}\biggr\}\le \frac{\log(2)\sqrt{2}}{\sqrt{2}-1}\qquad \qquad\\
\qquad + \max_{0< s'< s\le t}\max_{| z|\le \sqrt{2}/(\sqrt{2}-1)}\log
\mathbf{E}\exp\biggl\{z\lambda
\frac{\Delta^\zeta(s',s)}{\bar{\Delta}^v(s',s)}\biggr\},
\end{split}
\end{equation*}
where
$
\bar{\Delta}^v(s',s)=\sqrt{|v^2_s-v^2_{s'}|}.
$
\end{proposition}

\begin{definition}
 A zero mean process $\zeta_t,\ t\in \mathbb{R}^+$ is called ordered if there exists a continuous strictly monotone function
$v^2_t,\ t\in \mathbb{R}^+$ and some $\Lambda > 0$ such that
\begin{equation*}
\sup_{s',s\in \mathbb{R}^+:\,s' \neq s}\mathbf{E}\exp\biggl[ \Lambda \frac{\Delta^\zeta(s',s)}{\bar{\Delta}^v(s',s)}\biggr]<\infty.
\end{equation*}
\end{definition}

The next two propositions (see Lemmas 2 and 3  in \cite{G1}) show that the ordered process $\zeta_t$ can be controlled by the deterministic  function $v_t$.
\begin{proposition}\label{proposition.3}
Let $\zeta_t$ be an ordered process with $\zeta_0=0$. Then there exists a constant
$C(q',q)$   such that for all $1< q',q\le 2$, uniformly in $\free>0$
\begin{equation*}
\mathbf{E}\sup_{t\ge 0}\bigl[\zeta_t-\free v^q_t\bigr]_+^{q'} \le
\frac{C(q',q)}{\free^{q'/(q-1)}},
\end{equation*}
where $[x]_+=\max(0,x)$.
\end{proposition}

\begin{proposition}\label{proposition.4} Assume that there exists a monotone function $v_t, t\ge 0$ such that  a random process  $\zeta_t, \ t\in \mathbb{R}^+$,
satisfies
\begin{equation*}
\mathbf{E}\sup_{t\ge 0}\bigl[\zeta_t-\free v^q_t\bigr]_+^{q'} \le
\frac{C}{\free^{q'/(q-1)}},
\end{equation*}
for any $\free>0$  and some  $ q'\ge 1$, $q> 1 $.
Then   there exists a constant $C'$ such that for any random variable $\tau\in \mathbb{R}^+$ the following inequality holds
\begin{equation*}
\bigl[\mathbf{E}|\zeta_\tau|^{q'}\bigr]^{1/q'}\le C'\bigl[\mathbf{E}v^{q q'}_\tau\bigr]^{1/(q q')}.
\end{equation*}
\end{proposition}


In what follows, we focus on  typical ordered processes related to  the empirical risk.
The following two propositions (see Lemmas 4 and 5 in \cite{G1}) are
essential in controlling the cross term
\[
\sigma \sum_{k=1}^p \lambda^{-1/2}(k)[2h_\alpha(k)-h_\alpha^2(k)]\xi'(k)\beta(k)
\]
in the case, where $\alpha$ is a random variable depending on $\xi'(k),\, k=1,\ldots,p$.
\begin{proposition}\label{proposition.5}
 For any given  $\bar\alpha> 0$ and any $\free>0$,
\begin{equation*}
\begin{split}
&\mathbf{E}\sup_{0\le \alpha\le \alpha^\circ}\biggl\{\sum_{k=1}^p
\bigl[h_{\bar\alpha}(k)-h_\alpha(k)\bigr] b(k) \xi'(k)
\\&\qquad-\free\biggl[\sum_{k=1}^p
\bigl[h_{\bar\alpha}(k)-h_\alpha(k)\bigr]^2 b^2(k)\biggr]^{q'}\biggr\}_+^{q}
\le \frac{C}{\free^{q/(2q'-1)}}, \quad q'>1/2.
\end{split}
\end{equation*}
\end{proposition}

\begin{proposition}\label{proposition.6} Let  $\bar \alpha$ be a given smoothing parameter. Then for any $p\in [1,2)$,
 there exists a constant $C(p)$ so that for any data-driven smoothing parameter $\hat\alpha$,
\begin{equation*}
\begin{split}
&\mathbf{E} \biggl| \sum_{k=1}^p
\bigl[h_{\hat\alpha}(k)-h_{\bar\alpha}(k)\bigr]\lambda^{-1/2}(k) \beta(k)
\xi'(k)\biggr|^p\\&\quad  \le  C(p)
  \Bigl\{\mathbf{E}\max_{k}  \lambda^{-1}(k)
h_{\hat\alpha}^2(k)\Bigr\}^{p/2}\biggl[ \sum_{k=1}^p
\bigr[1-h_{\bar\alpha}(k)\bigl]^2\beta^2(k)\biggr]^{p/2}
\\&\qquad  +  C(p)\Bigl\{\max_{k}  \lambda^{-1}(k)
h_{\bar\alpha}^2(k)\Bigr\}^{p/2}\biggl[\mathbf{E} \sum_{k=1}^p
\bigr[1-h_{\hat\alpha}(k)\bigl]^2\beta^2(k)\biggr]^{p/2}.
  \end{split}
\end{equation*}
\end{proposition}

In order to obtain oracle inequalities in the case, where the noise variance is unknown, we will need the following lemma generalizing Proposition \ref{proposition.3}.

\begin{lemma}\label{lemma.1} Let
\[
\zeta_\alpha(b) =\sum_{k=1}^p [1-h_\alpha(k)]\xi'(k)b(k),\
v^2_\alpha(b) = \sum_{k=1}^p [1-h_\alpha(k)]^2 b^2(k), \
K=\frac{2}{(\sqrt{2}-1)^2}.
\]
 Then uniformly in $b\in \mathbb{R}^p$
\begin{equation*}
\mathbf{E}\exp\bigl\{ \sup_{\alpha\in \mathbb{R}^+}\bigl[\zeta_\alpha(b)-Kv_\alpha^2(b)\bigr]  \bigr\} \le C.
\end{equation*}
\end{lemma}

\begin{proof}
Since $h_\alpha(\cdot), \alpha \ge 0$, is the family of ordered functions,  it is not difficult to check that
\begin{equation}\label{eqn.21}
\mathbf{E}[\zeta_{\alpha'}(b)-\zeta_\alpha(b)]^2 \le |v^2_{\alpha'}(b)-v^2_{\alpha}(b)|.
\end{equation}
Indeed, we can rewrite (\ref{eqn.21}) in the following equivalent form
\begin{equation*}
\mathbf{E}\zeta_{\alpha'}(b)\zeta_\alpha(b)\ge \min\Bigl\{
v^2_{\alpha'}(b),v^2_{\alpha}(b)  \Bigr\}.
\end{equation*}
Assume  for definiteness that $h_\alpha(k)\ge h_{\alpha'}(k),\, k=1,2,\ldots,p$. Then    $1-h_\alpha(k)\le 1-h_{\alpha'}(k),\, k=1,2,\ldots,p$,
and we get
\begin{equation*}
\begin{split}
\mathbf{E}\zeta_{\alpha'}(b)\zeta_\alpha(b) = \sum_{k=1}^p [1-h_\alpha(k)] [1-h_{\alpha'}(k)] b^2(k)\\
\ge \sum_{k=1}^p  [1-h_{\alpha}(k)]^2 b^2(k)
=v_{\alpha}^2(b),
\end{split}
\end{equation*}
thus proving (\ref{eqn.21}).

Since $\zeta_\alpha(b)$ is a Gaussian process, we obtain
by (\ref{eqn.21})
\begin{equation}\label{eqn.22}
\log \mathbf{E}\exp\biggl\{\lambda \frac{\zeta_{\alpha'}(b)-\zeta_\alpha(b)}{\sqrt{|v^2_{\alpha'}(b)-v^2_{\alpha}(b)|}}\biggr\}\le \frac{\lambda^2}{2}.
\end{equation}

We may assume without loss of generality that $\sigma_\alpha$ is a continuous function in $\alpha\in \mathbb{R}^+$. Then let us fix some $\epsilon>0$ and define  $\alpha_l\in \mathbb{R}^+$ as roots of
equations
\[
v^2_{\alpha_l}(b)=(1+\epsilon)^{l}, \ l\ge 0.
\]
 Since $v_\alpha^2(b)\le \sum_{k=1}^p b^2(k)$,
the set of $\alpha_l$ is always finite but it may be  empty.

Let $\alpha^*$ be a root of the equation $v_{\alpha^*}^2(b)=1$.
Then by Proposition \ref{proposition.2} and (\ref{eqn.22}) we
obtain\vadjust{\eject}
\begin{equation*}
\begin{split}
&\mathbf{E}\exp\Bigl\{\max_{\alpha\in \mathbb{R}^+}\bigl[\zeta_\alpha(b) -Kv_\alpha^2(b)\bigr]\Bigr\}\\&\quad  \le \mathbf{E}\exp\Bigl\{\max_{\alpha>\alpha^* }\bigl[\zeta_\alpha(b) -Kv_\alpha^2(b)\bigr]\Bigr\} +\mathbf{E}\exp\Bigl\{\max_{\alpha\le \alpha^*}\bigl[\zeta_\alpha(b) -Kv_\alpha^2(b)\bigr]\Bigr\}
\\ &\quad \le \mathbf{E}\exp\Bigl\{\max_{\alpha\le \alpha^*}\zeta_\alpha(b) \Bigr\}+\sum_{l\ge 0}
\mathbf{E}\exp\Bigl\{\max_{\alpha_{l}<\alpha\le \alpha_{l+1}}\bigl[\zeta_\alpha(b) -Kv^2_{\alpha_{l-1}}(b)\bigr]\Bigr\}\\
 &\quad\le C+\sum_{l\ge 0}
\mathbf{E}\exp\biggl\{\max_{0<\alpha\le \alpha_{l}}\biggl[v_{\alpha_l}(b)\frac{\zeta_\alpha(b)} {v_{\alpha_l}(b)}
-Kv^2_{\alpha_{l-1}}(b)\biggr]\biggr\}\\
&\quad \le C+C \sum_{l\ge 0}
\exp\biggl\{\biggl[\frac{v_{\alpha_l}^2(b)} {(\sqrt{2}-1)^2}
-Kv^2_{\alpha_{l-1}}(b)\biggr]\biggr\}\\
 &\quad \le C+C\sum_{l\ge 0}
\exp\biggl\{\biggl[(1+\epsilon)^{l}\biggl(\frac{1} {(\sqrt{2}-1)^2}
-\frac{K}{1+\epsilon}\biggr)\biggr]\biggr\}\\
&\quad = C+C\sum_{l\ge 0}
\exp\biggl\{-(1+\epsilon)^{l-1}\frac{1-\epsilon} {(\sqrt{2}-1)^2}
\biggr\}.
\end{split}
\end{equation*}
This equation with $\epsilon =0.5$ completes the proof. 
\end{proof}

\subsection{Recovering the noise variance}
In this section, we focus on basic probabilistic properties of    the variance estimator
\[
\hat\sigma^2_{\hat\alpha} =\frac{\|Y-X\hat\beta_{\hat\alpha}\|^2}{\|1-H_{\hat\alpha}\|^2},
\]
in the case, where $\hat\alpha$ is a data-driven smoothing parameter.
We begin with a simple auxiliary fact.

\begin{lemma}\label{lemma.2}
  Let  $\eta'$ and  $\eta$ be nonnegative random variables. Then   the following inequality
\begin{align}\label{eqn.23}
\begin{split}
\mathbf{E}\eta'\eta^q \le\, &\frac{ 2^{q-1} \lambda^q }{(2-q)^q} \mathbf{E}\eta' \log^q
\biggl\{1+\frac{\eta'}{\mathbf{E}\eta'}\biggr\}\\
& {}+\frac{ 2^{q-1} \lambda^q}{(2-q)^q} \mathbf{E}\eta' \log^q \biggl\{1+\mathbf{E}
\biggl[\exp\biggl(\frac{\eta}{\lambda}\biggr)
-\frac{\eta}{\lambda}-1\biggr]\biggr\}
+q \lambda^q  \mathbf{E}\eta'
\end{split}
\end{align}
holds for  any $\lambda>0$ and $ q\in (1,2)$.
\end{lemma}

\begin{proof}
 Consider the following function
\[
F(z,y)\stackrel{\rm def}{=}\max_{x\ge 0}\Bigl\{x^qy-z[\exp(x)-1-x]\Bigr\}.
\]
Differentiating $x^qy-z[\exp(x)-1-x]$ in $x$,
it is easy to check that
\[
F(z,y)=x_*^qy-z[\exp(x_*)-1-x_*]\le x_*^qy,
\]
where $x_*$ is a root of the equation
\[
 x_*=\log\biggl(1+\frac{qy x_*^{q-1}}{z}\biggr).
\]
Since $\log(x)$ is  convex, it is clear
\[
\log\biggl(1+\frac{qy x_*^{q-1}}{z}\biggr)\le
\log \biggl(1+\frac{qy}{z}\biggr)+\biggl(1+\frac{qy}{z}\biggr)^{-1}
\frac{q(q-1)y}{z}(x^*-1).
\]
Therefore  $x_*\le x^*$, where $x^*$ is a root of the following linear equation
\[
x^*= \log \biggl(1+\frac{qy}{z}\biggr)+\biggl(1+\frac{qy}{z}\biggr)^{-1}
\frac{q(q-1)y}{z}(x^*-1).
\]
Since $q>1$, with a little algebra we get
\[
x^*\le  \biggl(1+\frac{qy}{z}\biggr)\biggl[1+\frac{q(2-q)y}{z}\biggr]^{-1}
 \log \biggl(1+\frac{qy}{z}\biggr)\le \frac{1}{2-q} \log \biggl(1+\frac{qy}{z}\biggr),
\]
thus arriving at the following upper bound
\[
F(z,y)\le \frac{y}{(2-q)^q} \log^q \biggl(1+\frac{qy}{z}\biggr).
\]

Now we are in a position to finish the proof. Notice
that for any $\lambda >0$
\begin{equation*}
\begin{split}
&\eta' \biggl(\frac{\eta}{\lambda}\biggr)^q -z\biggl[\exp\biggl(\frac{\eta}{\lambda}\biggr)
-1-\frac{\eta}{\lambda}\biggr]\\ &\quad \le \max_{x\ge 0}\biggl\{\eta' \biggl(\frac{x}{\lambda}\biggr)^q -z\biggl[\exp\biggl(\frac{x}{\lambda}\biggr)-1-\frac{x}{\lambda}\biggr]\biggr\}=F(z,\eta'),
\end{split}
\end{equation*}
and therefore
\begin{equation*}
\begin{split}
\mathbf{E}\eta'\eta^q \le\,& \lambda^q \biggl\{\mathbf{E}F(z,\eta')
+z\mathbf{E}\biggl[\exp\biggl(\frac{\eta}{\lambda}\biggr)
-1-\frac{\eta}{\lambda}\biggr]\biggr\}
\\ \le \,&\lambda^q \biggl\{\frac{1}{(2-q)^q}\mathbf{E}\eta'
 \log^q \biggl(1+\frac{q\eta'}{z}\biggr)+
z\mathbf{E}\biggl[\exp\biggl(\frac{\eta}{\lambda}\biggr)
-1-\frac{\eta}{\lambda}\biggr]\biggr\}.
\end{split}
\end{equation*}
Next, substituting in the above equation
\[
z=q \mathbf{E} \eta'  \biggr\{ \mathbf{E}\biggl[\exp\biggl(\frac{\eta}{\lambda}\biggr)
-1-\frac{\eta}{\lambda}\biggr]\biggr\}^{-1},
\]
we obtain
\begin{equation}\label{eqn.24}
\begin{split}
\mathbf{E}\eta'\eta^q  \le \lambda^q \biggl\{\frac{1}{(2-q)^q}\mathbf{E}\eta'
 \log^q \biggl(1+\frac{\eta'}{\mathbf{E}\eta'}\mathbf{E}\biggl[\exp\biggl(\frac{\eta}{\lambda}\biggr)
-1-\frac{\eta}{\lambda}\biggr]\biggr)+
q\mathbf{E}\eta'\biggr\}.
\end{split}
\end{equation}
Finally, applying  the following  inequality
\begin{equation*}
\begin{split}
\log^q(1+xy)\le [\log(1+y)+\log(1+x)]^q \le 2^{q-1}\log^{q
}(1+x)\\{}+2^{q-1}\log^q(1+y),\ x,y>0,
\end{split}
\end{equation*}
we get
\begin{equation*}
\begin{split}
\mathbf{E}\eta'
 \log^q \biggl\{1+\frac{\eta'}{\mathbf{E}\eta'}\mathbf{E}\biggl[\exp\biggl(\frac{\eta}{\lambda}\biggr)
-1-\frac{\eta}{\lambda}\biggr]\biggr\}
\quad \le  2^{q-1}\mathbf{E}\eta'
 \log^q \biggl(1+\frac{\eta'}{\mathbf{E}\eta'}\biggr)\\{}+
 2^{q-1}
 \log^q \biggl\{1+\mathbf{E}\biggl[\exp\biggl(\frac{\eta}{\lambda}\biggr)
-1-\frac{\eta}{\lambda}\biggr]\biggr\},
\end{split}
\end{equation*}
and combining this equation with (\ref{eqn.24}), we finish the proof of (\ref{eqn.23}).
\end{proof}

\begin{lemma}\label{lemma.3} Let $\eta$ be a  nonnegative sub-Gaussian random variable, i.e., such that for all
$\lambda >0$ and some $S>0$
\begin{equation}\label{eqn.25}
\mathbf{E}\exp(\eta/\lambda)\le C\exp(S^2/\lambda^2).
\end{equation}
Then for any $q\in [1,2)$
\begin{equation}\label{eqn.26}
\Bigl[\mathbf{E}\eta'^q\eta^q\Bigr]^{1/q}\le \frac{CS}{2-q}\biggl[ \mathbf{E}\eta'^q\log^q\biggl(1+\frac{\eta'^q}
{\mathbf{E}\eta'^q}\biggr)\biggr]^{1/q}.
\end{equation}
\end{lemma}

\begin{proof}
 Replacing  $\eta'$ in (\ref{eqn.23}) by $ \eta'^q$ and substituting (\ref{eqn.25}) in (\ref{eqn.23}), we get with $\lambda=S$
\begin{equation*}
\begin{split}
\mathbf{E}\eta'^q\eta^q
\le \frac{ 2^{q-1} S^q }{(2-q)^q} \mathbf{E}\eta'^q \log^q \biggl[1+\frac{\eta'^q}{\mathbf{E}\eta'^q}\biggr]
+\frac{ 2^{q-1} S^q}{(2-q)^q} \mathbf{E}\eta'^q
+q S^q  \mathbf{E}\eta'^q.
\end{split}
\end{equation*}
Let $F(x)=x\log^q(1+x)$. It is clear
that
\[
F'(x)=\log^{q}(1+x)+\frac{qx\log^{q-1}(1+x)}{1+x}
\]
is increasing in $x$ and therefore $F(x)$ is convex. Therefore by Jensen's inequality
\[
\mathbf{E}\eta'^q \log^q \biggl[1+\frac{\eta'^q}{\mathbf{E}\eta'^q}\biggr]
\ge \log(2)\mathbf{E}\eta'^q,
\]
and thus,  we arrive at (\ref{eqn.26}).  
\end{proof}

\begin{lemma}\label{lemma.4} Let
\[
\zeta_\alpha =\sum_{k=1}^p
\bigl[1-h_\alpha(k)\bigr]^2[1-\xi'^2(k)]
\]
and
\[
\Sigma_\alpha= 2 \|(1-h_\alpha)^2\|\sqrt{\log\log \frac{ \|(1-h_{\alpha^{\circ}})^2\|^2\exp(2)}{ \|(1-h_\alpha)^2\|^2}}.
\]
Then for any $s\in (1,2]$,
\begin{equation*}
\mathbf{P}\biggl\{\sup_{\alpha\le \alpha^{\circ}}\frac{\zeta_\alpha-s\Sigma_\alpha}{\|(1-h_\alpha)^2\|}\ge x\biggr\}\le \frac{C}{(s-1)^{3}}\exp\biggl\{-\frac{(3-s)^2x^2}{16}\biggr\}.
\end{equation*}
\end{lemma}

\begin{proof}
For some $\epsilon>0$  define $\alpha_k,\, k\ge 0$, as roots of equations
\[
\|(1-h_{\alpha_k})^2\|^2=(1+\epsilon)^{-k}\|(1-h_{\alpha^\circ})^2\|^2.
\]
Then, denoting for brevity
\[
G_{k+1}(x)= s\Sigma_{\alpha_{k+1}}+x \|(1-h_{\alpha_{k+1}})^2\|,
\]
we obtain
\begin{equation}\label{eqn.27}
\begin{split}
\mathbf{P}\, \biggl\{ \sup_{\alpha\le \alpha^{\circ}}\frac{\zeta_\alpha-s\Sigma_\alpha}{\|(1-h_\alpha)^2\|}\ge x \biggr\}\le \sum_{k=0}^\infty \mathbf{P}\, \biggl\{ \sup_{\alpha\in [\alpha_{k+1},\alpha_k]}\frac{\zeta_\alpha-s\Sigma_\alpha}{\|(1-h_\alpha)^2\|}\ge x \biggr\}\\ \quad   \le \sum_{k=0}^\infty \mathbf{P}\, \biggl\{ \sup_{\alpha\in [\alpha_{k+1},\alpha_k]}\zeta_\alpha\ge G_{k+1}(x) \biggr\}\\   \le   \sum_{k=0}^\infty \mathbf{P}\, \biggl\{ \zeta_{\alpha_{k+1}}\ge [1-f(\epsilon)]G_{k+1}(x)
 \biggr\}\\ \qquad   +
 \sum_{k=0}^\infty \mathbf{P}\, \biggl\{ \sup_{\alpha\in [\alpha_{k+1},\alpha_k]}[\zeta_\alpha-\zeta_{\alpha_{k+1}}]\ge f(\epsilon)G_{k+1}(x) \biggr\},
\end{split}
\end{equation}
where  $f(\epsilon)$ will be chosen later on.

Since $\log(1+x)\ge x -x^2/2,\ x\ge 0$, then for any $\lambda>0$
\begin{equation}\label{eqn.28}
\mathbf{E}\exp(\lambda \zeta_\alpha)\le \exp\bigl[\lambda^2\|(1-h_\alpha)^2\|^2\bigr],
\end{equation}
and
by the exponential Tchebychev inequality we get
\begin{equation}\label{eqn.29}
\begin{split}
\mathbf{P}\, \biggl\{ \zeta_{\alpha_{k+1}}\ge [1-f(\epsilon)]G_{k+1}(x) \biggr\} \le
\exp\biggl\{-\frac{[1-f(\epsilon)]^2 G_{k+1}^2(x)}{4\|(1-h_{\alpha_{k+1}})^2\|^2}\biggr\}\\ \quad  \le \exp\biggl\{-s^2[1-f(\epsilon)]^2\log[(k+1)\log(1+\epsilon)]-\frac{[1-f(\epsilon)]^2 x^2}{4}\biggr\}.
\end{split}
\end{equation}

To bound from above the last term in Equation (\ref{eqn.27}), we make use of  that $2h_\alpha(k)-h^2_\alpha(k)$ is a family of ordered functions, and thus
(see (\ref{eqn.21}))
\[
\|(1-h_{\alpha_{k}})^2-(1-h_{\alpha_{k+1}})^2\|^2\le \|(1-h_{\alpha_{k}})^2\|^2-\|(1-h_{\alpha_{k+1}})^2\|^2.
\]
Similarly to (\ref{eqn.28})
\begin{equation*}
\mathbf{E}\exp\bigl\{\lambda [\zeta_{\alpha_{k}}-\zeta_{\alpha_{k+1}}]\bigr\}\le \exp\bigl[\lambda^2\|(1-h_{\alpha_{k}})^2-(1-h_{\alpha_{k+1}})^2\|^2\bigr].
\end{equation*}
 Therefore with the help of Proposition \ref{proposition.2} and  the exponential Tchebychev inequality
 we obtain\vadjust{\eject}
\begin{equation}\label{eqn.30}
\begin{split}
\mathbf{P}\, \biggl\{ \sup_{\alpha_{k+1}<\alpha\le \alpha_k}[\zeta_\alpha-\zeta_{\alpha_{k+1}}]
\ge f(\epsilon) G_{k+1}(x) \biggr\}
\le \min_{\lambda>0}\exp\biggl\{-\lambda f(\epsilon) G_{k+1}(x)\\[3pt]
\qquad {}+ \frac{\lambda^2(\sqrt{2}-1)^2\|(1-h_{\alpha_{k}})^2-(1-h_{\alpha_{k+1}})^2\|^2}
{4[\|(1-h_{\alpha_{k}})^2\|^2-\|(1-h_{\alpha_{k+1}})^2\|^2]}\biggr\}\\[3pt]
 \quad \le
C\exp\biggl\{-\frac{(\sqrt{2}-1)^2f^2(\epsilon) G_{k+1}^2(x)}
{8[\|(1-h_{\alpha_{k}})^2\|^2-\|(1-h_{\alpha_{k+1}})^2\|^2]}\biggr\}\\[3pt]
\quad =
C\exp\biggl\{-\frac{(\sqrt{2}-1)^2s^2 f^2(\epsilon)}{4\epsilon}\log[(k+1)\log(1+\epsilon)]\\[3pt]
 \qquad \quad {}- \frac{(\sqrt{2}-1)^2x^2 f^2(\epsilon)}{8\epsilon}\biggr\}.
\end{split}
\end{equation}

Now we chose $f(\epsilon)$ to balance the exponents at the right-hand sides in (\ref{eqn.29}) and (\ref{eqn.30}), thus arriving at following equation for  this function
\[
\frac{(\sqrt{2}-1)^2f^2(\epsilon)}{2\epsilon}=[1-f(\epsilon)]^2.
\]
This yields
\[
f(\epsilon)=\frac{\sqrt{2\epsilon}}{\sqrt{2}-1+\sqrt{2\epsilon}}
.
\]
With this $f(\epsilon)$ and
 with (\ref{eqn.27}--\ref{eqn.30}) we get
\begin{equation*}\label{equ.21}
\begin{split}
\mathbf{P}\, \biggl\{ \sup_{\alpha\le \alpha^{\circ}}\frac{\zeta_\alpha-s\Sigma_\alpha}{\|(1-h_\alpha)^2\|} \ge x \biggr\}\le& \frac{C\exp\{-[1-f(\epsilon)]^2x^2/4\}}{\epsilon^{s^2[1-f(\epsilon)]^2}\bigl\{s^2[1-f(\epsilon)]^2-1\bigr\}_+}.
\end{split}
\end{equation*}

Finally, choosing $\epsilon $ as a root of
$
f(\epsilon)=(s-1)/2
$, we finish the proof. 
\end{proof}

We summarize the main properties of the variance estimator in the following lemma.
\begin{lemma}\label{lemma.5} For any $q \in (1,2)$
\begin{equation*}
\mathbf{E}\Bigl\{[\sigma^2-\hat\sigma_{\hat\alpha}^2]_+\sigma^{-2}
Pen(\hat\alpha)\Bigr\}^q \le  C(2-q)^{-q}\Ps^q(\alpha_\circ,\alpha^\circ)
\mathbf{E}\bigl[Pen(\hat{\alpha})\bigr]^q .
\end{equation*}
\end{lemma}

\begin{proof}
 By (\ref{eqn.8}) we obtain
\begin{equation}\label{eqn.31}
\begin{split}
\sigma^2-\hat\sigma^2_{\hat\alpha}=\frac{\sigma^2}{\|1-h_{\hat\alpha}\|^2}\sum_{k=1}^p \bigl[1-h_{\hat\alpha}(k)\bigr]^2[1-\xi'^2(k)]\\
{}-\frac{2\sigma}{\|1-h_{\hat\alpha}\|^2}\sum_{k=1}^p \bigl[1-h_{\hat\alpha}(k)\bigr]^2\xi'(k)\beta(k)\sqrt{\lambda(k)}\\
{}-\frac{1}{\|1-h_{\hat\alpha}\|^2}\sum_{k=1}^p \bigl[1-h_{\hat\alpha}(k)\bigr]^2\beta^2(k)\lambda(k).
\end{split}
\end{equation}

 The first term at the right-hand side in (\ref{eqn.31}) is  controlled with the help of Lemmas \ref{lemma.3} and \ref{lemma.4} (with $s=2$) as follows:
 \begin{equation}\label{eqn.32}
 \begin{split}
 \biggl\{\mathbf{E}\biggl|\frac{Pen(\hat{\alpha})}{\|1-h_{\hat\alpha}\|^2}\sum_{k=1}^p \bigl[1-h_{\hat\alpha}(k)\bigr]^2[1-\xi'^2(k)]\biggr|^q\biggr\}^{1/q}\\
\quad  =\biggl\{\mathbf{E}\frac{Pen^q(\hat{\alpha})}{\|1-h_{\hat\alpha}\|^q}\biggl|\frac{\zeta_{\hat{\alpha}}-2\Sigma_{\hat{\alpha}}+2\Sigma_{\hat{\alpha}}}{\|1-h_{\hat\alpha}\|}\biggr|^q\biggr\}^{1/q}
\\ \quad  \le \frac{C}{\|1-h_{\alpha_\circ}\|}\biggl\{\log\log\biggl(1+\frac{\|1-h_{\alpha_\circ}\|^2}{\|1-h_{\alpha^\circ}\|^2}\biggr)\biggr\}^{1/2}
 \bigl[\mathbf{E}Pen^q(\hat{\alpha})\bigr]^{1/q}
 \\ \qquad {}+\frac{C}{(2-p)\|1-h_{\alpha_\circ}\|}\log \biggl[1+\frac{Pen(\alpha_\circ)}{Pen(\alpha^\circ)}\biggr]
 \bigl[\mathbf{E}Pen^q(\hat{\alpha})\bigr]^{1/q}.
 \end{split}
 \end{equation}

To control the last two terms in (\ref{eqn.31}), notice that $\tilde{h}_\alpha(k)=
2h_\alpha(k)- h_\alpha^2(k),\ \alpha>0,$ is a family of ordered functions. Hence, applying Lemma \ref{lemma.1}  with
$$
b(k)=\frac{2\beta(k)\sqrt{\lambda(k)}}{K\sigma},
$$
we have
\begin{equation}\label{eqn.33}
\begin{split}
\mathbf{E}\exp\biggl\{\frac{2}{K\sigma^{2}} \biggl[2\sigma \sum_{k=1}^p \bigl[1-\tilde{h}_{\hat\alpha}(k)\bigr]\xi'(k)\beta(k)\sqrt{\lambda(k)}
\\  - \sum_{k=1}^p \bigl[1-\tilde{h}_{\hat\alpha}(k)\bigr]^2\beta^2(k)\lambda(k)
\biggr]\biggr\}\le C .
\end{split}
\end{equation}
This inequality and  Lemma \ref{lemma.2}
with
\begin{equation*}
\begin{split}
\eta'&=Pen^q(\alpha)\\
\eta&=\frac{2}{K\sigma^{2}} \biggl[2\sigma \sum_{k=1}^p \bigl[1-\tilde{h}_{\hat\alpha}(k)\bigr]\xi'(k)\beta(k)\sqrt{\lambda(k)}
 - \sum_{k=1}^p \bigl[1-\tilde{h}_{\hat\alpha}(k)\bigr]^2\beta^2(k)\lambda(k)
\biggr],
\end{split}
\end{equation*}
yield
\begin{equation}\label{eqn.34}
\begin{split}
\mathbf{E}\biggl[\frac{2\sigma}{\|1-h_{\hat\alpha}\|^2}\sum_{k=1}^p \bigl[1-h_{\hat\alpha}(k)\bigr]^2\xi'(k)\beta(k)\sqrt{\lambda(k)}\\
\qquad {}-\frac{1}{\|1-h_{\hat\alpha}\|^2}\sum_{k=1}^p \bigl[1-h_{\hat\alpha}(k)\bigr]^2\beta^2(k)\lambda(k)\biggr]^q
Pen^q(\hat{\alpha})\\
\quad \le \frac{C\sigma^{2q}}{(2-q)^q\|1-h_{\alpha_\circ}\|^{2q}}\log^q \biggl[1+\frac{Pen(\alpha_\circ)}{Pen(\alpha^\circ)}\biggr]\mathbf{E}Pen^q(\hat{\alpha}).
\end{split}
\end{equation}

Finally,  combining (\ref{eqn.31}), (\ref{eqn.32}), and (\ref{eqn.34}) and using  Jensen's inequality,
we finish the proof.
\end{proof}

\subsection{Proof of Theorem \ref{theorem.1}}

The following proposition (see Lemma 7 in \cite{G1})  summarizes   some  basic properties of the penalty defined by (\ref{eqn.13}--\ref{eqn.15}).
\begin{proposition}\label{proposition.7}
\begin{align}\nonumber
{Q^+(\alpha)}\ge D(\alpha)\max\biggl\{\sqrt{\log \frac{D(\alpha)}{D(\alpha^{\circ})}},\frac{1}{\mu_\alpha} \log \frac{D(\alpha)}{D(\alpha^{\circ})}\biggr\},
 \\
\nonumber
\mu_\alpha \ge \min\biggl\{\frac{1}{2}\sqrt{\log \frac{D(\alpha)}{D(\alpha^{\circ})}}, \frac14\biggr\},
\end{align}
If ${D(\alpha)}\ge \exp(2)D(\alpha^{\circ})$, then
\begin{align}
\nonumber
&{D(\alpha)}\ge {\mu_\alpha Q^+(\alpha)}\biggl[\log \frac{\mu_\alpha Q^+(\alpha)}{D(\alpha^{\circ})}\biggr]^{-1}.
\end{align}
For any $\alpha_1 \le \alpha_2$
\begin{align}\nonumber
&\frac{D(\alpha_1)}{D(\alpha_2)}\le \frac{Q^+(\alpha_1)}{Q^+(\alpha_2)}.
\end{align}
\end{proposition}

We begin the proof of Theorem \ref{theorem.1} with a simple generalization of Proposition~\ref{proposition.3}. Consider  the following  random process
\[
\eta_\alpha^\epsilon =\sum_{k=1}^p \lambda^{-1}(k)h_\alpha^\epsilon(k)[\xi'^2(k)-1],
\]
where $h_\alpha^\epsilon(k)=[2(1+\epsilon) h_{\alpha}(k)-\epsilon h_{\alpha}^2(k)]/(2+\epsilon)$.
\begin{lemma}\label{lemma.6} Let
 $q\in (1,2]$.
Then for any random variable $\hat\alpha\le \alpha^\circ$
\begin{equation*}
\mathbf{E}\eta_{\hat \alpha}^\epsilon \le \frac{C \sigma_{\alpha^\circ}^\epsilon}{\sqrt{q-1}} \biggl[\mathbf{E}\biggl(\frac{\sigma_{\hat\alpha}^\epsilon}
{\sigma_{\alpha^\circ}^\epsilon}\biggr)^{q}
\biggr]^{1/q},
\end{equation*}
where
\[
\sigma_\alpha^\epsilon=\biggl\{2\sum_{k=1}^p \lambda^{-2}(k)[h_\alpha^\epsilon(k)]^2
\biggr\}^{1/2}.
\]
\end{lemma}

\begin{proof}
 It is based on  the following fact. Let $S(x)=x^{1/(q-1)}, \ x \in \mathbb{R}^+$.
 Then
\begin{equation*}\label{equ.26}
\begin{split}
\rho(z)\,\stackrel{\rm def}{=}\, \mathbf{E}\sup_{\alpha\le \alpha^\circ}\biggl\{\eta_\alpha^\epsilon-z\sigma_{\alpha}^\epsilon S^{-1}\biggl(\frac{\sigma_\alpha^\epsilon}
{\sigma_{\alpha^\circ}^\epsilon}\biggr)
\biggr\}_+ \\ \le C\sigma_{\alpha^\circ}^\epsilon \int_0^\infty xS\biggl(\frac{x}{z}\biggr){\rm e}^{-Cx^2}\, dx,
\end{split}
\end{equation*}
where $S^{-1}(x)=x^{q-1}$ denotes the inverse function to $S(x)$.

To prove this inequality, define $\alpha_k, \ k=0,1,2,\ldots$  as roots of the following equations
\[
{\sigma_{\alpha_k}^\epsilon}=\sigma_{\alpha^\circ}^\epsilon S \bigl(1/z\bigr){\rm e}^k.
\]
Then,  noticing that $\eta_\alpha^\epsilon-\eta_{\alpha_k}^\epsilon$ is an ordered process,
we obtain by (\ref{eqn.12}) and Proposition \ref{proposition.2}
\begin{equation*}
\begin{split}
\rho(z) \le \mathbf{E} \sup_{\alpha_0\le \alpha\le \alpha^\circ}|\eta_\alpha^\epsilon|
+ \sum_{k=2}^\infty \mathbf{E}\sup_{\alpha_{k}\le \alpha < \alpha_{k-1} }\biggl\{\eta_\alpha^\epsilon-z \sigma_{\alpha_{k-1}}^\epsilon S^{-1}\biggl(\frac{\sigma_{\alpha_{k-1}}^\epsilon}
{\sigma_{\alpha^\circ}^\epsilon}
\biggr)\biggr]\biggr\}_+\\ \le C\sigma_{\alpha^\circ}^\epsilon S \biggl(\frac1z\biggr) \sum_{k=0}^\infty
{\rm e}^k \exp\Bigl\{-C\bigl[zS^{-1}\bigl(S \bigl(z^{-1}\bigr){\rm e}^k\bigl)\bigr]^2\Bigr\} \\
\le C\sigma_{\alpha^\circ}^\epsilon\int_{0}^\infty\exp\Bigl\{-C\bigl[z S^{-1}(u)\bigr]^2\bigr\}
\, du= C\sigma_{\alpha^\circ}^\epsilon \int_{0}^\infty{\rm e}^{-Cz^2 v^2}\, d S(v)
\\ \le  C\sigma_{\alpha^\circ}^\epsilon z^2\int_{0}^\infty S(v)v {\rm e}^{-Cz^2v^2}\, dv
=C\sigma_{\alpha^\circ}^\epsilon \int_0^\infty x S \biggl(\frac{x}{z}\biggr){\rm e}^{-Cx^2}\, dx.
\end{split}
\end{equation*}

Next we get  by the Laplace method
\begin{equation}\label{eqn.35}
\begin{split}
\rho(z)= C^{q/(q-1)} S\biggl(\frac1z\biggr)
\int_0^\infty x^{q/(q-1)}{\rm e}^{-x^2/2}\, dx\\
\le  C^{q/(q-1)} \biggl(\frac1z\biggr)^{1/(q-1)}\exp\biggl[\frac{q}{2(q-1)}\log \frac{q}{q-1}\biggr].
\end{split}
\end{equation}

To finish the proof, denote for brevity
\[
E=\mathbf{E}\biggl(\frac{\sigma_{\hat \alpha}^\epsilon}{\sigma_{\alpha^\circ}^\epsilon}\biggr)^q.
\]
Then by (\ref{eqn.35}) we obtain with a simple algebra
\begin{equation*}
\begin{split}
\mathbf{E}\eta_\alpha^\epsilon
\le  \min_{z}\biggl\{z \mathbf{E}\sigma_{\hat{\alpha}}^\epsilon \sigma_{\alpha^\circ}S^{-1}
\biggl(\frac{\sigma_{\hat\alpha}^\epsilon}
{\sigma_{\alpha^\circ}^\epsilon}\biggr)+ S\biggl(\frac{C}{z}\biggr)\exp\biggl[\frac{q}{2(q-1)}\log \frac{q}{q-1}\biggr]\biggr\}\\
 \\
 \le C\sigma_{\alpha^\circ}^\epsilon \min_{z}\biggl\{z E+ \biggl(\frac{C}{z}\biggr)^{1/(q-1)}\exp\biggl[\frac{q}{2(q-1)}\log \frac{q}{q-1}\biggr]\biggr\}\\ \le\frac{C}{\sqrt{q-1}}
\sigma_{\alpha^\circ}^\epsilon E^{1/q}.\\[-12pt]  
\end{split}
\end{equation*}
\end{proof}

The following important lemma provides an upper bound for $L_{\hat\alpha}(\beta)+(1+\gamma)Q^+(\hat\alpha)$.

\begin{lemma}\label{lemma.7}For  any data-driven $\hat\alpha$ and any given $\bar\alpha\in [\alpha_\circ,\alpha^\circ]$, the following inequality
\begin{equation*}
\begin{split}
&\Bigl\{\mathbf{E}\bigl[\sigma^{-2}L_{\hat\alpha}(\beta)+(1+\gamma)
Q^+(\hat\alpha)\bigr]^{1+\gamma/4}\Bigr\}^{1/(1+\gamma/4)}\\&\quad \le \frac{C}{[1-C \Ps(\alpha_\circ,\alpha^\circ)/\gamma]_+} \biggl[\frac{\bar{R}_{\bar\alpha}(\beta)}{\gamma \sigma^2}
+\frac{D(\alpha^\circ)}{\gamma^4}\biggr] 
\end{split}
\end{equation*}
holds uniformly in $\beta\in \mathbb{R}^p$ and $\gamma\in (0,1/{4})$.
\end{lemma}

\begin{proof}
 In view of the definition of $\hat\alpha$,
for any given smoothing parameter  $\bar\alpha$,  $R_{\hat\alpha}[Y,Pen]\le R_{\bar\alpha}[Y,Pen]$. It is easy to check with the help of (\ref{eqn.8}) that this inequality is equivalent to the following one
\begin{equation}\label{eqn.36}
\begin{split}
 L_{\hat\alpha}(\beta)+(1+\gamma)
\sigma^2Q^+(\hat\alpha)
 -\sigma^2 \sum_{k=1}^p \lambda^{-1}(k)\tilde{h}_{\hat\alpha}(k)[\xi'^2(k)-1]\\
 \qquad {}+2\sigma  \sum_{k=1}^p \lambda^{-1/2}(k)[1-h_{\hat\alpha}(k)]^2\xi'(k)\beta(k) +[\hat\sigma_{\hat\alpha}^2-\sigma^2]Pen(\hat\alpha)
\\ \quad \le
L_{\bar\alpha}(\beta)+(1+\gamma)
\sigma^2Q^+(\bar\alpha)
-\sigma^2 \sum_{k=1}^p \lambda^{-1}(k)\tilde{h}_{\bar\alpha}(k)[\xi'^2(k)-1]\\
\qquad {}+2\sigma  \sum_{k=1}^p \lambda^{-1/2}(k)[1-h_{\bar\alpha}(k)]^2\xi'(k)\beta(k) +[\hat\sigma_{\bar\alpha}^2-\sigma^2]Pen(\bar\alpha),
\end{split}
\end{equation}
where
$
\tilde{h}_\alpha(k)=2h_\alpha(k)-h^2_\alpha(k)
 $.
 We can rewrite (\ref{eqn.36}) as follows:
 \begin{equation}\label{eqn.37}
\begin{split}
 \frac{\gamma}{2}\bigl[L_{\hat\alpha}(\beta)+(1+\gamma)
\sigma^2Q^+(\hat\alpha)\bigr]\le
L_{\bar\alpha}(\beta)+(1+\gamma)
\sigma^2Q^+(\bar\alpha)
\\\quad {}-\sigma^2 \sum_{k=1}^p \lambda^{-1}(k)\tilde{h}_{\bar\alpha}(k)[\xi'^2(k)-1]\\
\quad {}+\sigma^2 \sum_{k=1}^p \lambda^{-1}(k)\tilde{h}_{\hat\alpha}(k)[\xi'^2(k)-1]
-\bigg(1+\frac{\gamma}{2}-\frac{\gamma^2}{2}\biggr)\sigma^2Q^+(\hat\alpha)\\
\quad {}+2\sigma  \sum_{k=1}^p \lambda^{-1/2}(k)[\tilde{h}_{\hat\alpha}(k)-\tilde{h}_{\bar\alpha}(k)]
\xi'(k)\beta(k)-\biggl(1-\frac\gamma2\biggr)L_{\hat\alpha}(\beta) \\ \quad+[\hat\sigma_{\bar\alpha}^2-\sigma^2]_+Pen(\bar\alpha)
 +[\sigma^2-\hat\sigma_{\hat\alpha}^2]_+Pen(\hat\alpha).
\end{split}
\end{equation}

Since $\bar\alpha$ is  given, we get by Jensen's inequality
\begin{equation}\label{eqn.38}
\begin{split}
\mathbf{E}\biggl|\sum_{k=1}^p \lambda^{-1}(k)\tilde{h}_{\bar\alpha}(k)[\xi'^2(k)-1]\biggr|^{1+\gamma/4}\le
C\biggl\{\sum_{k=1}^p \lambda^{-2}(k)\tilde{h}_{\bar\alpha}^2(k)\biggr\}^{1/2+\gamma/8}\\
\le C\bigl[D(\bar\alpha)\bigr]^{1+\gamma/4}
\le C\bigl[\sigma^{-2}\bar{R}_{\bar\alpha}(\beta)\bigr]^{1+\gamma/4}.
\end{split}
\end{equation}

 The third line in (\ref{eqn.37}) is bounded by Proposition \ref{proposition.1} as follows:
 \begin{equation}\label{eqn.39}
 \begin{split}
 \mathbf{E}\biggl[\sum_{k=1}^p \lambda^{-1}(k)\tilde{h}_{\hat\alpha}(k)[\xi'^2(k)-1]
-\bigg(1+\frac{\gamma}{2}-\frac{\gamma^2}{2}\biggr)Q^+(\hat\alpha)\biggr]^{
1+\gamma/4}_+\\
 \le\frac{CD^{1+\gamma/4}(\alpha^\circ)}{\gamma^3},
 \end{split}
 \end{equation}
  where $\gamma\le 1/\sqrt{2}$.

 The upper bound for  the fourth  line in (\ref{eqn.37}) is a little bit more tricky. Since $\bigl\{\tilde{h}_{\alpha}(\cdot), \alpha\in (0,\alpha^{\circ}]\bigr\}$ is a family of ordered functions, we obtain  by Proposition \ref{proposition.5} that for any $\epsilon>0$ and given $q'>1/2$,
\begin{equation}\label{eqn.40}
\begin{split}
 &\mathbf{E}\biggl|2\sigma\sum_{k=1}^p \lambda^{-1/2}(k)\bigl[\tilde{h}_{\hat\alpha}(k)-\tilde{h}_{\bar\alpha}(k)
\bigr]\xi'(k)\beta(k)\\&\quad - \epsilon\biggl\{\sigma^2\sum_{k=1}^p \lambda^{-1}(k)\bigl[\tilde{h}_{\hat\alpha}(k)-\tilde{h}_{\bar\alpha}(k)
\bigr]^2\beta^2(k)\biggr\}^{q'}\biggr|^{q} \le \frac{1}{(C\epsilon)^{q/(2q'-1)}}.
\end{split}
\end{equation}

To continue this inequality,  notice that
if $\hat\alpha\ge \bar \alpha$, then
\[
\frac{\tilde{h}_{\hat\alpha}(k)}{\tilde{h}_{\bar\alpha}(k)}\le 1, \quad \frac{\tilde{h}_{\hat\alpha}(k)}{\tilde{h}_{\bar\alpha}(k)}\ge \tilde{h}_{\hat\alpha}(k)
\]
and therefore
\begin{equation}\label{eqn.41}
\begin{split}
\sum_{k=1}^\infty
\bigl[\tilde{h}_{\bar\alpha}(k)-\tilde{h}_{\hat\alpha}(k)\bigr]^2\frac{\beta^2(k)}{\lambda(k)}
 =\sum_{k=1}^\infty
\tilde{h}^2_{\bar\alpha}(k)
\biggl[1-\frac{\tilde{h}_{\hat\alpha}(k)}{\tilde{h}_{\bar\alpha}(k)}\biggr]^2
\frac{\beta^2(k)}{\lambda(k)}
\\ \qquad
\le  \max_s \frac{\tilde{h}_{\bar\alpha}^2(s)}{\lambda(s)}  \sum_{k=1}^\infty
\bigl[1-\tilde{h}_{\hat\alpha}(k)\bigr]^2\beta^2(k)
\\ \qquad
\le  \max_s  \frac{4{h}_{\bar\alpha}^2(s)}{\lambda(s)}  \sum_{k=1}^\infty
\bigl[1-{h}_{\hat\alpha}(k)\bigr]^2\beta^2(k).
\end{split}
\end{equation}
Similarly, if $\hat\alpha < \bar\alpha$, then
\begin{equation}\label{eqn.42}
\begin{split}
\sum_{k=1}^\infty
\bigl[\tilde{h}_{\bar\alpha}(k)-\tilde{h}_{\hat\alpha}(k)\bigr]^2
\frac{\beta^2(k)}{\lambda(k)}
\le  \max_s \frac{ \tilde{h}_{\hat\alpha}^2(s) }{\lambda(s)} \sum_{k=1}^\infty
\bigl[1-\tilde{h}_{\bar\alpha}(k)\bigr]^2\beta^2(k)\\
\quad\le \max_s \frac{4 {h}_{\hat \alpha}^2(s)}{\lambda(s)}  \sum_{k=1}^\infty
\bigl[1-{h}_{\bar\alpha}(k)\bigr]^2\beta^2(k).
\end{split}
\end{equation}

So, combining (\ref{eqn.41}--\ref{eqn.42}) with  Young's inequality
\begin{equation*}
yx^s-x \le (1-s)s^{s/(1-s)}y^{1/(1-s)}, \quad
x,y\ge 0,\ s<1,
\end{equation*}
gives
\begin{equation}\label{eqn.43}
\begin{split}
\epsilon\biggl[\sigma^2\sum_{k=1}^\infty
\bigl[\tilde{h}_{\bar\alpha}(k)-\tilde{h}_{\hat\alpha}(k)\bigr]^2\lambda^{-1}(k)\beta^2(k)\biggr]^{q'}-\biggl(1-\frac{\gamma}{2}\biggr)L_{\hat\alpha}(\beta)\\
\quad \le\biggl(1-\frac{\gamma}{2}\biggr)^{-1}\epsilon\biggl[\sigma^2\sum_{k=1}^\infty
\bigl[\tilde{h}_{\bar\alpha}(k)-\tilde{h}_{\hat\alpha}(k)\bigr]^2\lambda^{-1}(k)\beta^2(k)\biggr]^{q'}-L_{\hat\alpha}(\beta)\\
\quad \le C\epsilon^{1/(1-q')}\biggl[\sum_{k=1}^\infty
\bigl[1-{h}_{\bar\alpha}(k)\bigr]^2\beta^2(k)+\sigma^2\max_k \frac{ {h}_{\bar\alpha}^2(k)}{\lambda(k)}\biggr]^{q'/(1-q')}.
\end{split}
\end{equation}

Thus, by (\ref{eqn.40}) and (\ref{eqn.43}) we obtain
\begin{equation*}
\begin{split}
\mathbf{E}\biggl|2\sigma  \sum_{k=1}^p \lambda^{-1/2}(k)\bigl[\tilde{h}_{\hat\alpha}(k)-\tilde{h}_{\bar\alpha}(k)
\bigr]\xi'(k)\beta(k)  - \biggl(1-\frac{\gamma}{2}\biggr)L_{\hat\alpha}(\beta)\biggr|^{q}\\
\le C\mathbf{E}\biggl|2\sigma  \sum_{k=1}^p \lambda^{-1/2}(k)\bigl[\tilde{h}_{\hat\alpha}(k)-\tilde{h}_{\bar\alpha}(k)
\bigr]\xi'(k)\beta(k)\\
\qquad - \biggl[\sigma^2\sum_{k=1}^p \lambda^{-1}(k)\bigl[\tilde{h}_{\hat\alpha}(k)-\tilde{h}_{\bar\alpha}(k)
\bigr]^2\beta^2(k)\biggr]^{q'}\biggr|^{q}\\
\qquad {}+ C \mathbf{E}\biggl| \biggl[\sigma^2\sum_{k=1}^p \lambda^{-1}(k)\bigl[\tilde{h}_{\hat\alpha}(k)-\tilde{h}_{\bar\alpha}(k)
\Bigr]^2\beta^2(k)\biggr]^{q'}- \biggl(1-\frac\gamma2\biggr) L_{\hat\alpha}(\beta)\biggr|^{q}
\\  \le (C\epsilon)^{-\frac{q}{2q'-1}}
 +C\epsilon^{\frac{q}{1-q'}}\biggl[\sum_{k=1}^\infty
\bigl[1-{h}_{\bar\alpha}(k)\bigr]^2\beta^2(k)+\sigma^2\max_k \frac{ {h}_{\bar\alpha}^2(k)}{\lambda(k)}\biggr]^{\frac{qq'}{1-q'}} .
\end{split}
\end{equation*}
Therefore, substituting in the above equation $q'=2/3$ and
\[
\epsilon =\biggl[\sum_{k=1}^\infty\bigl[1-{h}_{\bar\alpha}(k)\bigr]^2\beta^2(k)+\sigma^2\max_k \frac{ {h}_{\bar\alpha}^2(k)}{\lambda(k)}\biggr]^{-3q},
\] we get
\begin{equation}\label{eqn.44}
\begin{split}
\mathbf{E}\biggl|2\sigma  \sum_{k=1}^p \lambda^{-1/2}(k)\bigl[\tilde{h}_{\hat\alpha}(k)-\tilde{h}_{\bar\alpha}(k)
\bigr]\xi'(k)\beta(k)-\biggl(1-\frac\gamma2\biggr)L_{\hat\alpha}(\beta)\biggr|^{q}\\
\quad \le C\biggl[\sum_{k=1}^\infty
\bigl[1-{h}_{\bar\alpha}(k)\bigr]^2\beta^2(k)+\sigma^2\max_k \frac{ {h}_{\bar\alpha}^2(k)}{\lambda(k)}\biggr]^{q}.
\end{split}
\end{equation}

Now we proceed  with the last line in Equation (\ref{eqn.37}). Since $\bar\alpha$ is given,
we have by (\ref{eqn.31})
\begin{equation*}
\begin{split}
\bigl\{\mathbf{E}[\hat\sigma_{\bar\alpha}^2-\sigma^2]^2\bigr\}^{1/2}
\le \frac{\sigma^2}{\|1-h_{\bar\alpha}\|^2}\biggl\{\sum_{k=1}^p \bigl[1-h_{\bar\alpha}(k)\bigr]^4\biggr\}^{1/2}\\ {}+\frac{2\sigma}{\|1-h_{\bar\alpha}\|^2}\biggl\{\sum_{k=1}^p \bigl[1-h_{\bar\alpha}(k)\bigr]^4{\beta^2(k)}{\lambda(k)}\biggr\}^{1/2}\\ {}+\frac{1}{\|1-h_{\bar\alpha}\|^2}\sum_{k=1}^p \bigl\{1-h_{\bar\alpha}(k)\bigr\}^2{\beta^2(k)}{\lambda(k)}\\
\le \frac{C\sigma^2}{\|1-h_{\bar\alpha}\|}+\frac{C}{\|1-h_{\bar\alpha}\|^2}\sum_{k=1}^p \bigl[1-h_{\bar\alpha}(k)\bigr]^2{\beta^2(k)}{\lambda(k)}
\end{split}
\end{equation*}
and therefore
\begin{equation}\label{eqn.45}
\begin{split}
&\mathbf{E}\bigl\{[\hat\sigma_{\bar\alpha}^2-\sigma^2]_+Pen(\bar\alpha)\bigr\}^{1+\gamma/4} \le
C\bigl[\sigma^{-2} \Ps(\alpha_\circ,\alpha^\circ)\bar{R}_{\bar\alpha}(\beta)\bigr]^{1+\gamma/4}.
\end{split}
\end{equation}

The last term in (\ref{eqn.37}) can be bounded by Lemma \ref{lemma.5} and (\ref{eqn.16}) as follows:
\begin{equation}\label{eqn.46}
\begin{split}
&\mathbf{E}\bigl\{[\sigma^2-\hat\sigma_{\hat\alpha}^2]_+Pen(\hat\alpha)\bigr\}^{1+\gamma/4}\\ &\quad \le C\Ps^{1+\gamma/4}(\alpha_\circ) \mathbf{E}\bigl\{\sigma^{-2}L_{\hat\alpha}(\beta)+(1+\gamma)
Q^+(\hat\alpha)\bigr\}^{1+\gamma/4} .
\end{split}
\end{equation}

Finally combining Equations (\ref{eqn.37}), (\ref{eqn.38}), (\ref{eqn.39}), (\ref{eqn.44}), (\ref{eqn.45}), (\ref{eqn.46}),  we finish the proof.
\end{proof}

The next  idea    in   the proof of  Theorem \ref{theorem.1} is  that the data-driven parameter $\hat\alpha$ defined by (\ref{eqn.4}) cannot be very small, or equivalently, that the ratio
$D(\hat\alpha)/D(\alpha^{\circ})$ cannot be very large.
\begin{lemma}\label{lemma.8} For  any data-driven $\hat\alpha$ and any given $\bar\alpha\in [\alpha_\circ,\alpha^\circ]$,  the following upper bound holds
\begin{equation}\label{eqn.47}
\begin{split}
&\biggl\{\mathbf{E}\biggl[\frac {D(\hat\alpha)}{D(\alpha^{\circ})}\biggr]^{1+\gamma/4}\biggr\}^{1/(1+\gamma/4)}\\&\qquad\le \frac{C}{[1-C \Ps(\alpha_\circ,\alpha^\circ)/\gamma]_+}
 \R \biggl[\frac{\bar{R}_{\bar\alpha}(\beta)}{\sigma^2\gamma D(\alpha^{\circ})}+ \frac{1}{\gamma^4}\biggr],
\end{split}
\end{equation}
for any $  \gamma\in (0,1/4)$.
\end{lemma}

\begin{proof}
  Representing
 $$
 (1+\gamma)Q^+(\hat\alpha)=\biggl(1+\frac{\gamma}{2}\biggr)Q^+(\hat\alpha)
 +\frac{\gamma}{2} Q^+(\hat\alpha),
 $$
 we obtain with a simple algebra from (\ref{eqn.37})
\begin{equation*}
\begin{split}
\frac {\gamma\sigma^2}{2}\biggl[\sum_{k=1}^p \frac{h_{\hat\alpha }^2(k)}{ \lambda(k)} +(1+\gamma)Q^+(\hat\alpha)\biggr]
 \le \bar{R}_{\bar\alpha}(\beta) -\sigma^2 \sum_{k=1}^p \frac{\tilde{h}_{\bar\alpha}(k)}{\lambda(k)}[\xi'^2(k)-1]\\
\quad {}+\sigma^2\sup_{\alpha\le \alpha^{\circ}} \biggl[\sum_{k=1}^p \frac{\tilde{h}_{\alpha}(k)}{\lambda(k)}[\xi'^2(k)-1]
-\biggl(1+\frac{\gamma}{2}-\frac{\gamma^2}{2}\biggr)
Q^+(\alpha)\biggr]_+\\
\quad {}+2\sigma  \sum_{k=1}^p \lambda^{-1/2}(k)\bigl[\tilde{h}_{\hat\alpha}(k)-\tilde{h}_{\bar\alpha}(k)
\bigr]\xi'(k)\beta(k) -\biggl(1-\frac{\gamma}{2}\biggr)L_{\hat\alpha}(\beta) \\
\quad {}+[\hat\sigma_{\bar\alpha}^2-\sigma^2]Pen(\bar\alpha)+
[\sigma^2-\hat\sigma_{\hat\alpha}^2]_+Pen(\hat\alpha).
\end{split}
\end{equation*}

Combining this with Equations (\ref{eqn.38}), (\ref{eqn.39}), (\ref{eqn.44}), (\ref{eqn.45}), (\ref{eqn.46}), we obtain
\begin{equation}\label{eqn.48}
\begin{split}
&\biggl\{\mathbf{E} \biggl[\frac{\|h_{\hat\alpha}\|^2_\lambda +Q^+(\hat\alpha)}
{D(\alpha^\circ)}\biggr]^{1+\gamma/4}\biggr\}^{1/(1+\gamma/4)}
\\ & \qquad\le \frac{C}{[1-C \Ps(\alpha_\circ,\alpha^\circ)/\gamma]_+}  \biggl[\frac{C\bar{R}_{\bar\alpha}(\beta)}{\sigma^2 \gamma D(\alpha^\circ)}
+\frac{1}{\gamma^4}\biggr].
\end{split}
\end{equation}

To continue this inequality, we need a lower bound for $\|h_\alpha\|_\lambda^2 +Q^+(\alpha)$.
Notice that
\[
f(x)\stackrel{\rm def}{=}F(x)-\frac{x^2}{1-2x}=\frac{1}{2}\log(1-2x)+x+\frac{x^2}{1-2x}
\]
is a non-negative function for $x\ge 0$ since
\[
f'(x)=\frac{2x^2}{(1-2x)^2}\ge 0 \text{ and }f(0)=0.
\]
 Therefore the following inequality holds
\begin{equation}\label{eqn.49}
F(x)\ge \frac{x^2}{1-2x}.
\end{equation}

Let
 \begin{equation}\label{eqn.50}
 k_\alpha =\arg\max _k \frac{h_\alpha(k)}{\lambda(k)},
 \end{equation}
 then by
 (\ref{eqn.14}) and (\ref{eqn.49}) we obviously get
\begin{equation*}
\log\frac{D(\alpha)}{D(\alpha^\circ)}\ge F[\mu_\alpha \rho_\alpha(k_\alpha)]\ge \frac{[\mu_\alpha \rho_\alpha(k_\alpha)]^2}{1-2 \mu_\alpha \rho_\alpha(k_\alpha) }.
\end{equation*}
With this inequality we obtain
\[
\mu_\alpha \rho_\alpha(k_\alpha)\le \biggl\{1+\biggl[1+\log^{-1}\frac{D(\alpha)}{D(\alpha^\circ)}\biggr]^{1/2}\biggr\}^{-1},
\]
thus arriving at
\begin{equation}\label{eqn.51}
\mu_\alpha^{-1} \ge 2\rho_\alpha(k_\alpha).
\end{equation}

 Now we are in a position to bound from below $ \|h_\alpha\|_\lambda^2 +Q^+(\alpha)$. By (\ref{eqn.13}--\ref{eqn.15}), (\ref{eqn.50}--\ref{eqn.51}), and (\ref{eqn.17}) we obtain
\begin{equation}\label{eqn.52}
\begin{split}
\|h_\alpha\|_\lambda^2 +Q^+(\alpha)\ge \|h_\alpha\|_\lambda^2+\frac{2D(\alpha)}{\mu_\alpha}\sum_{k=1}^p \frac{[\mu_\alpha\rho_\alpha(k) ]^2}{1-2\mu_\alpha\rho_\alpha(k)}\\
\ge \|h_\alpha\|_\lambda^2+\frac{D(\alpha)}{\mu_\alpha}\sum_{k=1}^p F[\mu_\alpha\rho_\alpha(k)] =
\|h_\alpha\|_\lambda^2+\frac{D(\alpha)}{\mu_\alpha}\log\frac{D(\alpha)}{D(\alpha^\circ)}\\
\ge \|h_\alpha\|_\lambda^2+2\rho_\alpha(k_\alpha)D(\alpha)\log\frac{D(\alpha)}{D(\alpha^\circ)}\\
 \ge  \|h_\alpha\|_\lambda^2+ \frac{h_\alpha(k_\alpha)}{\lambda(k_\alpha)}\log\frac{D(\alpha)}{D(\alpha^\circ)}
\\
 \ge  \|h_\alpha\|_\lambda^2+ \log\frac{D(\alpha)}{D(\alpha^\circ)}\max _k \frac{h_\alpha(k)}{\lambda(k)}
\ge CD(\alpha)\log\frac{D(\alpha)}{D(\alpha^\circ)}.
\end{split}
\end{equation}

With the help of
(\ref{eqn.52}) we continue (\ref{eqn.48}) as follows:
\begin{equation}\label{eqn.53}
\begin{split}
\biggl\{\mathbf{E} \biggl[\frac {D(\hat\alpha)}{D(\alpha^{\circ})}\log\frac {D(\hat\alpha)}{D(\alpha^{\circ})}\biggr]^{1+\gamma/4}\biggr\}^{1/(1+\gamma/4)}\qquad\qquad
\\[6pt] \quad \le \frac{C}{[1-C \Ps(\alpha_\circ,\alpha^\circ)/\gamma]_+}  \biggl[\frac{C\bar{R}_{\bar\alpha}(\beta)}{\sigma^2 \gamma D(\alpha^\circ)}
+\frac{1}{\gamma^4}\biggr].
\end{split}
\end{equation}
To control from below the left-hand side in the above equation, notice that
\begin{equation}\label{eqn.54}
\begin{split}
\mathbf{E}\biggl[\frac {D(\hat\alpha)}{D(\bar{\alpha})}\log\frac {D(\hat\alpha)}{D(\bar{\alpha})}
\biggr]^{1+\gamma/4}=\frac{1}{(1+\gamma/4)^{1+\gamma/4}}
\mathbf{E}\biggl[\frac {D(\hat\alpha)}{D(\bar{\alpha})}\biggr]^{1+\gamma/4}
\\[6pt]
\times \biggl\{\log\biggr[\frac {D(\hat\alpha)}{D(\bar{\alpha})}\biggr]^{1+\gamma/4}\biggr\}^{1+\gamma/4}.
\end{split}
\end{equation}

To finish the proof, let us consider the function $f(x)=x\log^{1+\gamma/4}(x)$, $ x\ge 1$. Computing its second order derivative, one can easily check that $f(x)$ is convex for all $x\ge \exp(1)={\rm e}$. So, $f(x+{\rm e}-1)$ is convex for $x\ge 1$. It is easily seen
there exists a constant $C>0$ such that for all $x\ge 1$
\[
f(x)\ge \frac{1}{2} f(x+{\rm e}-1)-C.
\]
Therefore by (\ref{eqn.54}) and Jensen's inequality,
\begin{equation}\label{eqn.55}
\begin{split}
\mathbf{E}\biggl[\frac {D(\hat\alpha)}{D(\bar{\alpha})}\log\frac {D(\hat\alpha)}{D(\bar{\alpha})}\biggr]^{1+\gamma/4}\ge C \biggl\{\mathbf{E}\biggl[\frac {D(\hat\alpha)}{D(\bar{\alpha})}\biggr]^{1+\gamma/4}+{\rm e}-1\biggr\}
\\[6pt]
\quad \times \log^{1+\gamma/4}\biggl\{\mathbf{E}\biggr[\frac {D(\hat\alpha)}{D(\bar{\alpha})}\biggr]^{1+\gamma/4}+{\rm e}-1\biggr\}-C.
\end{split}
\end{equation}

Let
\[
\psi(x)=(x+{\rm e}-1)\log^{1+\gamma/4}(x+{\rm e}-1).
\]
It is easy to check  that the inverse function $\psi^{-1}(x)$ satisfies the following inequality
\begin{equation*}\label{equ.46}
\psi^{-1}(x)\le (x+{\rm e}-1)\log^{-1-\gamma/4}(x+{\rm e}-1).
\end{equation*}
Therefore combining this equation and (\ref{eqn.55}) with (\ref{eqn.53}), we arrive at (\ref{eqn.47}).
\end{proof}

Now we are ready to proceed with the proof of Theorem \ref{theorem.1}.
Let $\epsilon>0$ be a small given number to  be defined later on. By  (\ref{eqn.8}) and (\ref{eqn.9}),
the following equation for the skewed excess risk\vadjust{\vspace*{1cm}\eject}
\begin{equation}\label{eqn.56}
\begin{split}
\mathcal{E}(\epsilon)\stackrel{\rm def}{=} \sup_{\beta \in
\mathbb{R}^p}\mathbf{E}_\beta\Bigl\{\|\beta -\hat
\beta_{\hat\alpha}\|^2 -(1+\epsilon) \bigl\{{R}_{\hat \alpha}[Y]+\mathcal{C}\bigr\} \Bigr\}\\
 \ =\sup_{\beta \in
\mathbb{R}^p}\mathbf{E}_\beta\biggl\{ -\epsilon
\sum_{k=1}^p\bigl[1-h_{\hat\alpha}(k)\bigr]^2 \beta^2(k)-\epsilon\sigma^2 \sum_{k=1}^p
\lambda^{-1}(k) h_{\hat\alpha}^2(k)\\ \quad- (1+\epsilon)(1+\gamma)\sigma^2 Q^+(\hat\alpha)
\\
\quad -2\sigma
\sum_{k=1}^p\bigl\{1+\epsilon-[(1+2\epsilon)h_{\hat\alpha}(k)-\epsilon
h_{\hat\alpha}^2(k)
]\bigr\}\beta(k)\lambda^{-1/2}(k)
\xi'(k)  \\ \quad  +\sigma^2\sum_{k=1}^p
\lambda^{-1}(k)\bigl[2(1+\epsilon) h_{\hat\alpha}(k)-\epsilon h_{\hat\alpha}^2(k)\bigr][\xi'^2(k)-1]
\\ \qquad
 +[\sigma^2-\hat\sigma_{\hat\alpha}^2]Pen(\hat\alpha)
\biggr\}
\end{split}
\end{equation}
holds.

We proceed with the second   line from below at the right-hand side of this display.
 By Lemmas \ref{lemma.6} and \ref{lemma.8}, we obtain
\begin{equation}\label{eqn.57}
\begin{split}
\sigma^2\mathbf{E}\sum_{k=1}^p
\lambda^{-1}(k) \bigl[(1+\epsilon) h_{\hat\alpha}(k)-\epsilon h_{\hat\alpha}^2(k)\bigr][\xi'^2(k)-1]\\
\le \frac{C}{[1-C \Ps(\alpha_\circ,\alpha^\circ)/\gamma]_+\sqrt{\gamma} }
 \R \biggl[\frac{\bar{R}_{\bar\alpha}(\beta)}{\sigma^2\gamma D(\alpha^{\circ})}+ \frac{1}{\gamma^4}\biggr].
\end{split}
\end{equation}

The next step is to bound the third  line from below at the right-hand side of (\ref{eqn.56}).
It suffices   to note that $\tilde{h}_\alpha^\epsilon (k)=\bigl[(1+2\epsilon) h_{\alpha}(k)-\epsilon h_{\alpha}^2(k)\bigr]/(1+\epsilon)$
is the family of ordered functions. Hence, by Proposition \ref{proposition.6}, we get with $\bar\alpha=\arg\min_{\alpha\in[\alpha_\circ,\alpha^\circ]}\bar{R}_\alpha(\beta)$
 \begin{equation}\label{eqn.58}
\begin{split}
2\sigma\mathbf{E}
\sum_{k=1}^p\bigl\{1+\epsilon-[(1+2\epsilon)h_{\hat\alpha}(k)-\epsilon
h_{\hat\alpha}^2(k)
]\bigr\}\beta(k)\lambda^{-1/2}(k)
\xi'(k)\\ \quad =2(1+\epsilon)\sigma\mathbf{E}
\sum_{k=1}^p\bigl[h_{\bar\alpha}^\epsilon(k)-h_{\hat\alpha}^\epsilon(k)\bigr]\beta(k)\lambda^{-1/2}(k)
\xi'(k)\\ \quad  \le C \biggl[\sigma^2\mathbf{E}\max_k \lambda^{-1}(k) {h}_{\hat \alpha}^2(k)  \sum_{k=1}^p
\bigl[1-{h}_{\bar\alpha}(k)\bigr]^2\beta^2(k)\bigg]^{1/2}\\
\qquad {}+C \biggl[\sigma^2\max_k \lambda^{-1}(k) {h}_{\bar \alpha}^2(k) \mathbf{E} \sum_{k=1}^p
\bigl[1-{h}_{\hat\alpha}(k)\bigr]^2\beta^2(k)\bigg]^{1/2}\\ \quad \le
C\sigma^2\epsilon^{-1}\mathbf{E}\max_k \lambda^{-1}(k) {h}_{\hat \alpha}^2(k)+\epsilon
\sum_{k=1}^p
\bigl[1-{h}_{\bar\alpha}(k)\bigr]^2\beta^2(k)
\\ \qquad {}+
C\sigma^2\epsilon^{-1}\max_k \lambda^{-1}(k) {h}_{\bar \alpha}^2(k)+\epsilon
\mathbf{E}\sum_{k=1}^p
\bigl[1-{h}_{\hat\alpha}(k)\bigr]^2\beta^2(k)
.
\end{split}
\end{equation}

Therefore, substituting (\ref{eqn.45}), (\ref{eqn.46}), (\ref{eqn.57}), (\ref{eqn.58}) in (\ref{eqn.56}), we obtain the following upper bound for the skewed excess risk
\begin{equation}\label{eqn.59}
\begin{split}
\mathcal{E}(\epsilon)\le C\sigma^2\epsilon^{-1}\mathbf{E}\max_k \lambda^{-1}(k) {h}_{\hat \alpha}^2(k)-\sigma^2\mathbf{E}Q^+(\hat\alpha) +C \Ps(\alpha_\circ,\alpha^\circ)\bar{R}_{\bar\alpha}(\beta)\\  +C\sigma^2\epsilon^{-1}\max_k \lambda^{-1}(k) {h}_{ \bar\alpha}^2(k)+\epsilon \sum_{k=1}^p
\bigl[1-{h}_{\bar\alpha}(k)\bigr]^2\beta^2(k)\\ {}+ \frac{C\sigma^2D(\alpha^\circ)}{[1-C \Ps(\alpha_\circ,\alpha^\circ)/\gamma]_+\sqrt{\gamma} }
 \R \biggl[\frac{\bar{R}_{\bar\alpha}(\beta)}{\sigma^2\gamma D(\alpha^{\circ})}+ \frac{1}{\gamma^4}\biggr].
\end{split}
\end{equation}

Let us consider the function
\[
U(\epsilon)=\max_{\alpha\le \alpha^{\circ}}\bigl\{C\epsilon^{-1}\max_k \lambda^{-1}(k) {h}_{ \alpha}^2(k)-Q^+(\alpha)\bigr\}.
\]
Since
\begin{equation*}
\begin{split}
\max_k \frac{{h}_{ \alpha}^2(k)}{\lambda(k)} \le\max_k \frac{{h}_{ \alpha}(k)}{\lambda(k)}\le \biggl[\sum_{k=1}^p \frac{{h}_{ \alpha}^2(k)}{\lambda^2(k)}\biggr]^{1/2}\\
\le \biggl\{\sum_{k=1}^p \frac{{h}_{ \alpha}^2(k)}{\lambda^2(k)}[2-h_\alpha(k)]^2\biggr\}^{1/2}
\le  \frac{D(\alpha)}{\sqrt{2}}
\end{split}
\end{equation*}
 and by Proposition \ref{proposition.7}  $$Q^+(\alpha)\ge D(\alpha)\sqrt{\log \frac{D(\alpha)}{D(\alpha^{\circ})}},$$ we get
\begin{equation*}
\begin{split}
U(\epsilon)\le D(\alpha^{\circ}) \max_{\alpha\le \alpha^{\circ}}\biggl\{\frac{C}{\epsilon}\frac{D(\alpha)}{D(\alpha^{\circ})}-
\frac{D(\alpha)}{D(\alpha^{\circ})}\biggl[\log\frac{D(\alpha)}{D(\alpha^{\circ})}\biggr]^{1/2}\biggr\}\\
\le  D(\alpha^{\circ}) \max_{x\ge 1}\biggl\{ \frac{Cx}{\epsilon}-x\sqrt{\log(x)}\biggr\}.
\end{split}
\end{equation*}
One can easily check with a simple algebra that
\begin{equation}\label{eqn.60}
\max_{x\ge 1}\biggl\{ \frac{Cx}{\epsilon}-x\sqrt{\log(x)}\biggr\}\le\frac{\epsilon}{C}
\exp\biggl[\frac{C^2}{\epsilon^2}\biggr].
\end{equation}
Indeed, let $x^*=\arg\max_x\Bigl\{Cx/\epsilon-x\sqrt{\log(x)}\Bigr\}$. Then, differentiating $Cx/\epsilon-x\sqrt{\log(x)}$ in $x$, we obtain the following equation for $x^*$
\[
\frac{C}{\epsilon}-\sqrt{\log(x^*)}-\frac{1}{2\sqrt{\log(x^*)}}=0.
\]
Therefore
\[
x^*=\exp\biggl\{\biggl(\frac{C}{2\epsilon}+\sqrt{\frac{C^2}{4\epsilon^2}-1}\biggl)^2\biggr\}\le \exp\biggl(\frac{C^2}{\epsilon^2}\biggr).
\]
This equation proves (\ref{eqn.60}) since
\[
\max_{x\ge 1} \biggl\{\frac{Cx}{\epsilon }-x\sqrt{\log(x)}\biggr\}\le \frac{Cx^*}{\epsilon}.
\]
With (\ref{eqn.60}) we continue  (\ref{eqn.59}) as follows:
\begin{equation*}\label{equ.61}
\begin{split}
\mathcal{E}(\epsilon)\le C\sigma^2D(\alpha^{\circ}){\epsilon}
\exp\frac{C^2}{\epsilon^2}+C\Ps(\alpha_\circ,\alpha^\circ)\bar{R}_{\bar\alpha}(\beta)\\ {}+C\sigma^2\epsilon\sum_{k=1}^p \lambda^{-1}(k) {h}_{ \bar\alpha}^2(k)+\epsilon \sum_{k=1}^p
\bigl[1-{h}_{\bar\alpha}(k)\bigr]^2\beta^2(k)\\ {}+
\frac{C \sigma^2D(\alpha^\circ)}{[1-C \Ps(\alpha_\circ,\alpha^\circ)/\gamma]_+\sqrt{\gamma} }
 \R \biggl[\frac{\bar{R}_{\bar\alpha}(\beta)}{\sigma^2\gamma D(\alpha^{\circ})}+ \frac{1}{\gamma^4}\biggr]\\
 \le C\sigma^2D(\alpha^{\circ}){\epsilon}
\exp\frac{C^2}{\epsilon^2}+C\epsilon\bar{R}_{\bar\alpha}(\beta) +C \Ps(\alpha_\circ,\alpha^\circ)\bar{R}_{\bar\alpha}(\beta)\\ {}+ \frac{C}{[1-C \Ps(\alpha_\circ,\alpha^\circ)/\gamma]_+\sqrt{\gamma} }
 \R \biggl[\frac{\bar{R}_{\bar\alpha}(\beta)}{\sigma^2\gamma D(\alpha^{\circ})}+ \frac{1}{\gamma^4}\biggr] .
\end{split}
\end{equation*}

Therefore, substituting this equation in
 \[
 \mathbf{E}\|\beta-\hat\beta_{\hat\alpha}\|^2 \le (1+\epsilon)\bar{R}_{\bar\alpha}(\beta)+\mathcal{E}(\epsilon),
 \]
  we get
\begin{equation}\label{eqn.61}
\begin{split}
\mathbf{E}\|\beta-\hat\beta_{\hat\alpha}\|^2 \le [1+C \Ps(\alpha_\circ,\alpha^\circ)]r(\beta)+
C\sigma^2D(\alpha^{\circ})\times \\ \times \inf_{\epsilon}\biggl[\epsilon\exp\frac{C^2}{\epsilon^2}
+\frac{\epsilon r(\beta)}{\sigma^2 D(\alpha^{\circ})}\biggr]\\ {}+
\frac{C \sigma^2D(\alpha^\circ) }{[1-C \Ps(\alpha_\circ,\alpha^\circ)/\gamma]_+\sqrt{\gamma} }
 \R \biggl[\frac{r(\beta)}{\sigma^2\gamma D(\alpha^{\circ})}+ \frac{1}{\gamma^4}\biggr].
\end{split}
\end{equation}

Hence, to finish the proof of the theorem, it remains to check that
\begin{equation}\label{eqn.62}
\inf_\epsilon F(\epsilon,u)=\inf_{\epsilon}\biggl[\epsilon\exp\frac{C^2}{\epsilon^2}
+\epsilon u\biggr]\le \frac{Cu}{\sqrt{\log(u)}}.
\end{equation}
Let $\epsilon_*=\arg\min_{\epsilon}F(\epsilon,u)$. Then, differentiating $F(\epsilon,u)$ in $\epsilon$, we arrive at the following equation for $\epsilon_*$
\[
\exp\biggl(\frac{C^2}{\epsilon_*^2}\biggr)-\frac{C^2}{\epsilon_*^2}\exp\biggl(\frac{C^2}{\epsilon_*^2}\biggr)+u=0.
\]
Thus
\[
\frac{C^2}{\epsilon_*^2}+\log\biggl(\frac{C^2}{\epsilon_*^2}-1\biggr)=u
\]
and it follows immediately from the above equation that
\[
\epsilon_*\le \frac{C}{\sqrt{\log(u)}}
\]
and therefore
\[
F(\epsilon_*,u)\le 2u\epsilon_*\le \frac{2Cu}{\sqrt{\log(u)}},
\]
thus proving (\ref{eqn.62}).

Finally, substituting (\ref{eqn.62}) with $u=r(\beta)/[\sigma^2 D(\alpha^\circ)]$ in (\ref{eqn.61}), we complete the proof of the theorem.

\end{document}